\documentclass[a4paper]{amsart}
\usepackage{amsmath}
\usepackage{amssymb}
\usepackage{enumitem}
\DeclareRobustCommand{\rchi}{{\mathpalette\irchi\relax}}
\newcommand{\irchi}[2]{\raisebox{\depth}{$#1\chi$}} 

\newcommand{\NN}{\mathbb{N}}
\newcommand{\RR}{\mathbb{R}}

\newcommand{\ZZ}{\mathbb{Z}}

\newtheorem{theorem}{Theorem}[section]
\newtheorem{proposition}[theorem]{Proposition}
\newtheorem{lemma}[theorem]{Lemma}
\newtheorem{corollary}[theorem]{Corollary}
\newtheorem{remar}[theorem]{Remark}

\newenvironment{remark}{\begin{remar} \rm}{\end{remar}}
\newtheorem{remars}[theorem]{Remarks}

\newtheorem{defin}[theorem]{Definition}
\newenvironment{definition}{\begin{defin} \rm}{\end{defin}}
\newtheorem{defins}[theorem]{Definitions}
\newenvironment{definitions}{\begin{defins} \rm}{\end{defins}}
\newtheorem{examp}[theorem]{Example}
\newenvironment{example}{\begin{examp} \rm}{\end{examp}}
\newtheorem{notat}[theorem]{Notation}
\newenvironment{notation}{\begin{notat} \rm}{\end{notat}}
\newtheorem{notats}[theorem]{Notations}
\newenvironment{notations}{\begin{notats} \rm}{\end{notats}}
\title{Key polynomials, separate and immediate valuations, and simple extensions of valued fields}
\author{G\'erard Leloup}
\address{G\'erard Leloup, 
Laboratoire Manceau de Math\'ematiques, 
LMM EA 3263, 
Le Mans Universit\'e, 
avenue Olivier Messiaen, 
72085 LE MANS CEDEX, 
FRANCE.}
\email{gerard.leloup-At-univ-lemans.fr}
\keywords{simple extension, valuation, key polynomial.}
\subjclass[2010]{12J10, 12J20, 12F99.}
\date{May 16, 2022}
\begin{document}
\begin{abstract} We give a first-order definition of key polynomials, we show the links with previous definitions, 
that it is relevant to study key degrees, and to use a kind of valuations that we call partially multiplicative. 
We also prove or reprove several properties. 
\end{abstract}
\maketitle
\indent \\
\indent We have changed definitions from the first version of this paper, in order to 
make them consistent with those of other papers. Some computational and tedious proofs have been omitted in this 
new version. The reader can find them in versions 1 and 2 of this paper.  \\
\indent Valuations occur in different areas of mathematics, for example in algebraic geometry. 
In particular, some authors studied the extension of a valuation $\nu$, defined on a field $K$, 
to $K[\rchi]$. For this purpose, they construct $K$-module valuations which converges to $\nu$. 
These $K$-module valuations are defined using polynomials which are called key polynomials. 
For his part, the author of the present article has studied the first order properties of 
extensions of valued fields. In this context, vs defecless extensions, immediate extensions and the 
initial segments $\nu(l-K)$ ($l\notin K$) play an important role. Now, there are similarities 
between these two types of studies. 
We found in\-te\-res\-ting to revisit the key polynomials and the associated valuations in the light of another background. 
The approach of the present paper relies on euclidean division of polynomials, bases of 
vector spaces or modules, and the notions of vs defectless and immediate extensions (Definitions \ref{def11} and \ref{def14}). 
Indeed, vs defectless (or separate) and immediate extensions play a crucial role in the study of 
the first-order theory of extensions of valued fields (\cite{B1}, \cite{B2}, \cite{D88}, \cite{D91}, 
\cite{L89}, \cite{L03}). \\
\indent Let $(K,\nu)$ be a valued field and $\rchi$ be algebraic or transcendental over $K$. Key polynomials were introduced to 
approximate an extension of $\nu$ to the field $K(\rchi)$ 
by $K$-module valuations. In particular, this study uses $K$-module 
valuations such that there exists a monic polynomial $\Phi$ of degree $d$ such that the family $(\Phi^n)$ ($n\geq 0$) 
is separated over the $K$-module $K_{d-1}[\rchi]$ of polynomials of degree at most $d-1$ (they are also called augmented 
valuations or truncated mappings). Here we denote them by $\nu_{\Phi}$. In the case where $\rchi$ is 
algebraic over $K$, these $K$-module valuations cannot be valuations of fields since they do not 
satisfy the rule $\nu(fg)=\nu(f)+\nu(g)$. However, they can satisfy this rule if $\mbox{deg}(f)+\mbox{deg}(g)<
[K(\rchi)\!:\!K]$. If this holds, then we call them pm valuations. In fact, $\nu_{\Phi}$ is a pm valuation if 
$\Phi$ is what we call here a weak key polynomial (in short w key polynomial). 
These w key polynomials are related to the key polynomials 
defined by MacLane (we show that $\Phi$ being a w key polynomial for $\nu$ is equivalent to being a MacLane 
key polynomial for $\nu_{\Phi}$), but we do not require to define families of w key polynomials by induction. \\
\indent 
In order to construct families of pm valuations which approximate $\nu$ on $K(\rchi)$, 
we can restrict to a subfamily of w key polynomials whose definition differs slightly. 
We call them key polynomials, because we prove that these polynomials are the 
key polynomials defined in \cite{NS18}, \cite{DMS18} 
and \cite{N19} (in \cite{DMS18} they are called abstract key polynomials). 
Here, we do not restrict to transcendental extensions and 
we give a characterization of key polynomials by means of first order formulas 
(in the language of valued fields together 
with a predicate interpreted by the generator of the simple extension). \\
\indent 
We classify the key polynomials according to their degrees, that we call key degrees. On the one hand, 
key degrees $d$ such that the set $\nu(\rchi^d-K_{d-1}[\rchi])$ has a greatest element, on the other hand 
key degrees $d$ such that the set $\nu(\rchi^d-K_{d-1}[\rchi])$ has a no greatest element. This already appeared 
in the works of S.\ MacLane and of M.\ Vaqui\'e.  
They construct the key polynomials by induction, starting from the degree $1$ key polynomials, and 
two cases arise: the set $\nu(\rchi^d-K_{d-1}[\rchi])$ having or not a greatest element. 
We will show that in the first case we have a 
vs defectless subextension, which is generated by a valuation independent element, 
and in the last one this subextension is immediate. So in the first case we say that $d$ is an independent key degree, 
and in the latter case we say that it is an immediate key degree. 
In the construction of key polynomials by induction on the degree, some authors  introduced 
the definition of limit key polynomials. In fact, a key polynomial is a limit key polynomial 
if, and only if, the preceding key degree is an immediate one. \\
\indent  In the present paper we do not require the reader having any previous knowledge of key polynomials. 
As much as possible, we try to use only elementary properties of valuations. 
However, we use the graded algebra associated to a valuation in two proofs. 
Indeed, these graded algebras are a good tool for providing nice proofs of some 
properties of vs defectless extensions. With exception of Section \ref{subsec37}, this paper is self-contained. \\[2mm]
{\bf General properties.}\\
\indent A {\it valuation} on a field $K$ is a morphism $\nu$ from the multiplicative group ($K^*,\cdot)$ to an 
abelian linearly ordered group $(\nu K,+)$ called the {\it valuation group}, which enjoys an additional property. 
We add an element $\infty$ to $\nu K$ ($\infty > \nu K$), we 
set $\nu(0):=\infty$, and we assume that 
$\nu(x+y)\geq \min (\nu(x),\nu(y))$, for every $x$, $y$ in $K$. It follows that 
if $\nu(x)\neq \nu(y)$, then  $\nu(x+y)= \min (\nu(x),\nu(y))$. If $\nu(x)=\nu(y)$, then 
$\nu (x+y)$ can be arbitrarily large. The set $\{x\in K\mid \nu(x)\geq 0\}$ is a local domain and 
$\{x\in K\mid \nu(x)> 0\}$ is its maximal ideal, they are called respectively the {\it valuation ring} 
and the {\it maximal ideal} of the valued field. The quotient of the 
valuation ring by the maximal ideal is a field which  is denoted by $K_{\nu}$, and is called the {\it residue field} of 
the valued field. The {\it residue cha\-rac\-te\-ris\-tic} of $(K,\nu)$ is the 
cha\-rac\-te\-ris\-tic of $K_{\nu}$. We have ${\rm char}\; K_{\nu}=0\Rightarrow {\rm char}\; K=0$. 
For $x$ in the valuation ring 
of $(K,\nu)$, its class modulo the maximal ideal will be denoted by $x_{\nu}$. 
For every subset $M$ of $K$ we denote by $\nu(M)$, or $\nu M$, the set $\{\nu(x)\mid x\in M\}\backslash\{\infty\}$. \\
\indent 
We assume that $\nu$ is a finite rank valuation (the {\it rank} of $\nu$ is the number of proper 
convex subgroups of $\nu K$), 
so that $\nu K$ has countable cofinality. 
If the rank is $1$, then $\nu K$ embeds in the ordered group $(\RR,+)$; we also say that 
$\nu$ is {\it archimedean}. \\
\indent 
An extension $(L|K,\nu)$ of valued fields consists in an extension $L|K$ of fields, 
where $L$ is equipped with a valuation $\nu$. If $L|K$ is a field extension and $\nu$ is a valuation 
on $K$, then there is a least one extension of $\nu$ to $L$. \\
\indent Assume that $M$ is a $K$-module (where $K$ is a field). A mapping $\nu$ from $M$ 
to a linearly ordered group together with an element $\infty$ will be called a 
{\it $K$-module valuation}, if for every $y_1$, $y_2$ in $M$ and $x\in K$: 
$\nu(y_1)=\infty \Leftrightarrow y_1=0$, $\nu(y_1+y_2)\geq \min (\nu(y_1),\nu(y_2))$, 
$\nu(xy_1)=\nu(x)+\nu(y_1)$ (see.\ \cite{FVK}). It follows that its restriction to $K$ 
is a valuation of field.\\
\indent Now, let $(K(\rchi)|K,\nu)$ be a simple extension of valued fields. 
If there is no confusion, then for every polynomial $f(\rchi)$ 
of $K[\rchi]$ we write $f$ instead of $f(\rchi)$. Let $\nu_i$ be a $K$-module valuation on $K[\rchi]$ such that 
$\forall f\in K[\rchi]\; \nu_i(f)\leq \nu(f)$. 
Assume that $\rchi$ is algebraic (so that $K(\rchi)=K[\rchi]$), and that there exists 
$f\in K[\rchi]$ such that $\nu_i(f)<\nu(f)$. By hypothesis, we have $\nu_i(1/f)\leq 
\nu(1/f)$. Hence $\nu_i(f)<\nu(f)=-\nu(1/f)\leq -\nu_i(1/f)$. So $\nu_i(1/f)\neq -\nu_i(f)$. It follows that 
$\nu_i$ doesn't satisfy the rule $\nu_i(fg)=\nu_i(f)+\nu_i(g)$. Now, we will prove in Proposition 
\ref{avprop32} that, in some cases, if ${\rm deg}(f)+{\rm deg}(g)<[K[\rchi]\!:\! K]$, then 
$\nu_i(fg)=\nu_i(f)+\nu_i(g)$. This motivates the following definitions. \\
\indent 
Let $\nu$ be a $K$-module valuation. 
We say that $\nu$ is  {\it partially multiplicative}, or a  {\it pm valuation}, 
if for every $f$, $g$ in $K[\rchi]$ such that $\mbox{deg}(f)+\mbox{deg}(g)<[K[\rchi]\!:\! K]$ we have 
$\nu(fg)=\nu(f)+\nu(g)$. In the case where 
$[K[\rchi]\!:\! K]=\infty$ (i.e.\ $\rchi$ is transcendental), 
we assume in addition that $g\neq 0$ implies $\nu(f/g)=\nu(f)-\nu(g)$, that is, 
$\nu$ is a field valuation in the usual sense. If $\nu$ is a field valuation (in the usual sense), then   
we will also say that it is {\it  multiplicative}.  \\[2mm]
{\bf Summary of the paper.}\\
\indent 
In Section \ref{section1} we generalize the definitions of immediate and dense extensions to 
extensions of modules, equipped with $K$-module valuations. We also review definitions and properties of se\-pa\-ra\-ted sequences 
and vs defectless extensions. 
Then we focus on the case of simple extensions. We characterize immediate simple extensions by the condition that the sets 
$\nu(\rchi^n-K_{n-1}[\rchi])$ have no greatest elements (Proposition \ref{prop111}), and we characterize vs defectless
simple extensions by the condition that the sets 
$\nu(\rchi^n-K_{n-1}[\rchi])$ haves greatest element (Proposition \ref{prop214}). 
Then we recall definitions and properties of the graded algebras associated to valuations. 
In particular, we note that in the case of a generalized power series field, this graded algebra is 
isomorphic to the ring of ge\-ne\-ra\-li\-zed polynomials. 
Furthermore, we prove that a family is separated if, and 
only if, its image in the graded algebra is linearly independent (Proposition \ref{fact220}). 
In Section \ref{section3} we characterize w key polynomials, key polynomials and key 
degrees. In particular, we note that if $\Phi$ is an irreducible polynomial, then it is a w key polynomial if, and only if, 
the quotient field $K[\rchi]/(\Phi)$ is a valued field. We define 
the vs defectless $K$-module valuations $\nu_{\phi}$. We deduce other characterizations of w key polynomials. 
These characterizations will be used in latter sections. 
In Section \ref{subsection32} we compare the definition of w key polynomials with the definition of S.\ MacLane and M.\  Vaqui\'e 
(Proposition \ref{prop320}). 
We set properties of key degrees in Section \ref{section4}. We show that in the case where $\nu(\rchi^d-K_{d-1}[\rchi])$ has a 
maximum, we only need to look at monic polynomials with maximal valuations. 
We characterize key degrees, that we call valuational, by means of the value group (Proposition \ref{avpropdivb}). 
We also focus on conditions for existing only 
a finite number of key degrees, since this makes simpler the constructions of families of pm valuations which approximate 
$\nu$. 
We look at the case of a dense key degree, that is immediate key degrees $d$ 
such that $\nu(\rchi^d-K_{d-1}[\rchi])$ is equal to $\nu(K[\rchi])\backslash\{\infty\}$ (Proposition \ref{prop49}).  
We also give characterizations of the successor of a given key degree (Theorems \ref{thm4xx} and \ref{propdiv}). 
These cha\-rac\-te\-ri\-za\-tions will be useful for proving the equivalence of our definition and the definition of 
abstract key polynomials. In Subsection \ref{subsection34} we give a construction of a 
family of pm valuations which approximate $\nu$ (defined by a family of key polynomials), 
without proceeding by induction on the degree (Theorem \ref{thm46}). 
In Section \ref{subsec37} we show that the key 
polynomials defined in \cite{NS18}, \cite{DMS18} 
and \cite{N19} are the same as the key polynomials defined here (Corollaries \ref{cor351} 
and \ref{cor353}). \\[2mm]
\indent Before starting, let us set some notations. 
Let $l  \in  L$ and $M$ be a $K$-submodule of $L$. We denote by  $\nu(l-M)$ the subset 
$\{ \nu(l-x) \mid x \in  M\}\backslash\{\infty\}$ of $\nu L$. 
For any polynomial $f$, we denote by 
$\nu(f(M))$, or $\nu f(M)$, the subset $\{\nu (f(x))\mid x\in M\}\backslash \{\infty\}$. \\
\indent Note that $\nu(l-M)\cap \nu M$ is an initial segment of $\nu M$. 
\section{Immediate and vs defectless extensions}\label{section1}
In order to show the importance of vs defectless and immediate extensions of valued fields 
as well of the initial segments $\nu(l-K)$, we start by stating some results of model theory 
of extensions of valued fields. 
A valued field $(K,\nu)$ is said to be {\it algebraically maximal} if no 
extension of $\nu$ to an algebraic extension of $K$ is immediate. If the residue characteristic 
is $0$, then being algebraically maximal is equivalent to 
being {\it henselian}, i.e. $\nu$ having a unique extension to any agebraic extension of $K$. 
A famous theorem of J.\ Ax, S.\ Kochen and Y.\ Ershov (\cite{AK65}) says that the elementary 
theory of a henselian valued 
field $(K,\nu)$ of residue characteristic $0$ is determined by the elementary theory of its residue field 
and the elementary theory of its value group. Next, this result was extended to other families of 
algebraically maximal valued fields. 
The aim is to get similar results in the case of extensions of valued fields. 
Now, in \cite{D91} F.\ Delon proved that given a theory $T_F$ of fields of 
characteristic $0$ and a theory $T_V$ of non trivial linearly ordered abelian groups, 
the theory of immediate henselian extensions $(L|K,\nu)$, where $K_{\nu}$ is a model of $T_F$ and 
$\nu K$ is a model of $T_V$, is indecidable and admits $2^{\aleph_0}$ completions.    
This theory also depends at least on the family of all initial segments $\nu(l-K)$, with $l\in L\backslash K$.
Now, if $K$, $L$ are henselian, char$K_{\nu}=0$, and $(L|K,\nu)$ is vs defectless, 
then the first-order theory of 
the extension is determined by the theories of the residual extension and of the 
extension of valued groups. The same holds if $(L|K,\nu)$ is an extension of 
algebraically maximal Kaplansky fields or of real-closed fields 
(see \cite{B1}, \cite{B2}, \cite{L89}, \cite{L03}). 
Furthermore, we have similar results with dense extensions. 
In the case where $(L|K,\nu)$ is an 
extension of valued fields of residue characteristic $0$, there exists a henselian subfield 
$H$, $K\subseteq H\subseteq L$, such that $(H|K,\nu)$ is vs defectless and $(L|H,\nu)$ is 
immediate (see \cite{D88}). This shows that it can be interesting to 
focus on vs defectless and immediate extensions. \\[2mm]
\indent In this section, $L|K$ is an extension of fields and $\nu$ is a $K$-module valuation on $L$. 
\subsection{Immediate extensions} 
If $M$ is a $K$-module, then we assume that $\nu M$ has no greatest element. This holds if $\nu K$ is not 
trivial. Indeed, for every $x\in K$ with $\nu(x)>0$ and $y\in M$, we have $xy\in M$ and 
$\nu(xy)=\nu(x)+\nu(y)>\nu(y)$.
\begin{definitions}\label{def11} 
Let  $M\subseteq N$ be $K$-submodules of $L$, and $l\in L$. \\
\indent 
We say that $l$ is {\it pseudo-limit} over $(M,\nu)$ if $\nu(l-M)\subseteq \nu M$, and $\nu(l-M)$ has no 
maximal element. \\
\indent 
We say that $l$  is {\it limit} over $(M,\nu)$ if $\nu(l-M) =\nu M$.\\
\indent 
The extension $(N|M,\nu)$ is said to be {\it immediate} if every element of $N$ is pseudo-limit over 
$(M,\nu)$. \\
\indent 
The extension $(N|M,\nu)$ is said to be 
 {\it dense} if every element of $N$ is limit over $M$. \end{definitions}
\indent It follows that if $(N|M,\nu)$ is dense, then it is immediate. \\ 
\indent The proof of the following is left to the reader (see for example \cite{L18}). \\
\indent Let $l \in L$ and $M$ be a $K$-submodule of $L$. 
The element $l$ is pseudo-limit over $(M,\nu)$ if, and only if, for every $x\in M$ there exists $y\in M$ such that 
$\nu(l-y)>\nu(l-x)$. 
\begin{notations} For $\gamma$ in $\nu L$ and $M$ a $K$-submodule of $L$, 
let $M_{\gamma,\nu}$ be the $K_{\nu}$-module 
$\{ x\in M\mid \nu(x)\geq \gamma\}/\{x\in M\mid \nu(x)>\gamma\}$. In the case where $\gamma=0$, 
we often write $M_{\nu}$ instead od $M_{0,\nu}$. \\
For $f\in M$ with $\nu(f)\geq \gamma$, we denote by $f_{\gamma,\nu}$ the class of $f$ modulo 
$\{g\in M\mid \nu(g)>\gamma\}$. 
\end{notations}
\begin{remark}\label{extimptim}. \\
\indent 
1) $(N|M,\nu)$ is immediate if, and only if, $\nu N=\nu M$ and, for every $\gamma \in \nu M$, 
$N_{\gamma,\nu}=M_{\gamma,\nu}$. 
  \\
\indent 
2) $(N|K,\nu)$ is immediate  if, and only if, $\nu N=\nu K$ and $N_{\nu}=K_{\nu}$. \end{remark}
\begin{proof} Left to the reader (see for example \cite{L18}). \end{proof} 
\subsection{Vs defectless extensions}\label{subsec21} 
\begin{definitions}\label{def14} (\cite{B1}, \cite{B2}, \cite{BCK18}).  
Let $M$ be a $K$-submodule of $L$. \\ 
\indent 
1) A sequence $(l_1,\dots,l_n)$ of $L$ is said to be {\it separated over} $M$ (or 
$\nu$-separated if necessary) if for every 
$x_{1}, \dots , x_{n}$ in $M$, we have: 
$\nu (x_{1}l_{1}+ \cdots + x_{n}l_{n}) = \min_{1 \leq i \leq n} 
\nu (x_{i}l_{i})$. If $M=K$, then we say separated instead of separated over $K$. 
An infinite sequence is called {\it separated} if every finite subsequence is separated. \\
\indent 
2) The extension $(L|K,\nu)$ is said to be 
{\it vs defectless} if every finitely generated $K$-submodule of $L$ admits a basis which is separated over $K$. 
If this holds, then we say that $\nu$ is {\it vs defectless} (or 
{\it vs defectless over} $K$). \end{definitions}
\indent Note that if $(l_1,\dots,l_n)$ is separated, then $l_1,\dots,l_n$ are linearely independent over $K$. \\[2mm]
%
\indent In \cite{B1} and \cite{B2} Baur introduces the terms separated families and separated extensions. The authors of 
\cite{BCK18} noted that the word separated can bring confusion in algebraic geometry (for example with separable extension). 
So they prefer use the terms valuation independent sets and vs defectless extensions. Following these authors we use the term 
vs defectless extension instead of separated extension. Concerning Baur's separated families, 
there is a slight difference with the definition of valuation independent families. A valuation independent set is separated, 
but if $B$ is valuation independent, then $B\cup\{ 1\}$ remains separated. 
\begin{example}\label{examp217} Let $(K,\nu)$ be a valued field. Pick some $x$ in $K$, 
some $\gamma$ in 
an extension of $\nu K$, and, for every $x_0,x_1,\dots,x_n$ in $K$, 
set $\nu'(x_n(\rchi-x)^n+\cdots +x_1(\rchi-x)+x_0) :=\min (\nu(x_n)+n\gamma,\dots, \nu(x_1)+
\gamma,\nu(x_0))$. Then one can check that $\nu'$ defines a pm valuation on 
the ring $K[\rchi]$. The pm valuation $\nu'$ is vs defectless 
and $1,(\rchi-x),\dots,(\rchi-x)^n,\dots$ is a separated basis of $(K[\rchi],\nu')$. 
\end{example}
\indent Note that if $\nu''$ is another $K$-module valuation on $K[\rchi]$ which extends $\nu$ and such that 
$\nu(\rchi-x)=\gamma$, then, for every $f$ in $K[\rchi]$, $\nu'(f)\leq \nu''(f)$. \\[2mm]
\indent The following properties provide other examples of vs defectless extensions. \\
$\bullet$ If $(K,\nu)$ is a maximal valued field, then every 
multiplicative extension of $(K,\nu)$ is vs defectless (\cite[Lemma 3]{B1}).\\
$\bullet$ Assume that $\nu$ is multiplicative on $L$, that the residue characteristic of $(K,\nu)$ is $0$ 
and that $(K,\nu)$ is {\it henselian} that is, $\nu$ admits a unique extension to every 
algebraic extension of $K$. Then any algebraic extension 
of $(K,\nu)$ is vs defectless (\cite[Corollaire 7]{D88}). \\
$\bullet$ Assume that 
$L$ is a finite algebraic extension of $K$. Then $(L|K,\nu)$ is 
vs defectless if, and only if, the $K$-module $L$ admits a separated basis. This is an immediate consequence of 
the following lemma. 
\begin{lemma} (\cite[Lemma 5]{D88})\label{D88lm5}. Let $M \subseteq N$ be two  
$K$-submodules of $L$ such that $M$ is finitely generated and $N$ admits a 
separated basis. Then $M$ admits a separated basis. \end{lemma}
\indent The remainder of this subsection is dedicated to properties which make clear the difference between 
immediate and vs defectless extensions. \\[2mm]
\indent Assume that $\nu$ is multiplicative on $L$ 
and that $(K,\nu)$ is henselian of residue characteristic $0$. Then $(L|K,\nu)$ is vs defectless 
if, and only if, $L$ is linearly disjoint over $K$ from every immediate extension of $(K,\nu)$ (\cite[p.\ 421]{D88}). 
\begin{theorem}\label{caracsep} (\cite[p.\ 421]{D88}, \cite{BCK18}) Let $N$ be a $K$-submodule of $L$. Then, 
$(N|K,\nu)$ is a vs defectless extension if, and only if, for every finitely generated 
$K$-submodule $M$ of $N$ and $l\in N\backslash M$, the set $\nu(l-M)$ has a maximal element. 
\end{theorem}
This theorem has been stated in \cite[p.\ 421]{D88}, assuming that $(K,\nu)$ is henselian and 
char$(K_{\nu})=0$. A general proof has been given independently in \cite{BCK18} and \cite{L18}. \\
\indent We know that if $L|K$ is finite and $\nu$ is multiplicative, then $1\leq 
[L_{\nu}\!:\! K_{\nu}](\nu L\!:\!\nu K)\leq [L\!:\! K]$. Furthermore, by Remark \ref{extimptim}, 
$(L|K,\nu)$ is immediate if, and 
only if, $1=[L_{\nu}\!:\! K_{\nu}](\nu L\!:\!\nu K)$. The following theorem shows that $(L|K,\nu)$ 
being vs defectless can be seen as the opposite case. It has been proved  independently in \cite{BCK18} and \cite{L18}.
\begin{theorem}\label{relatdegsep} Assume that 
$L|K$ is a finite algebraic extension of fields and that $\nu$ is multiplicative on $L$. 
Then $(L|K,\nu)$ is vs defectless if, and only if, $[L\!:\! K]=[L_{\nu} \!:\! K_{\nu}](\nu L\!:\!\nu K)$. 
\end{theorem}
\indent Let $(L|K,\nu)$ be a finite extension of valued fields (where $\nu$ is 
multiplicative). We say that $(L|K,\nu)$ is {\it defectless} if 
$\displaystyle{\frac{[L\!:\! K]}{[L_{\nu} \!:\! K_{\nu}](\nu L\!:\!\nu K)}}$ is equal to the number 
of extensions of $\nu_{|K}$ to $L$. \\ 
\indent As a corollary of Theorem \ref{relatdegsep} we get that 
 if $(K,\nu)$ is henselian, then every defectless finite algebraic extension of $(K,\nu)$ is 
vs defectless. 
\subsection{Extensions generated by one element}\label{subsec22} 
In this subsection $\rchi$ is algebraic or transcendental over $K$. Recall that if $d$ is a non-negative integer, 
then $K_d[\rchi]$ denotes the $K$-module of polynomials of degree at most $d$. 
Note that $K_0[\rchi]=K$ and in the transcendental case we let $K_{\infty}[\rchi]=K[\rchi]$. 
\begin{proposition}\label{prop111} Let $d<d'$ in $\NN\cup \{\infty\}$ such that 
if $\rchi$ is algebraic over $K$, then $d'<[K[\rchi]\!:\! K]$. \\ 
1) $(K_{d'}[\rchi]|K_d[\rchi],\nu)$ is immediate if, and only if, for every integer $n\in \{d+1,\dots,d'\}$, 
$\nu(\rchi^n-K_{n-1}[\rchi])$ has no maximal element. \\
2) Assume that $\nu$ is a pm valuation and that $\nu K_d[\rchi]$ is a subgroup of $\nu(K[\rchi])$. Then the extension 
$(K_{d'}[\rchi]|K_d[\rchi],\nu)$ is dense if, and only if, for every integer $n\in \{d+1,\dots,d'\}$, 
$\nu(\rchi^n-K_{n-1}[\rchi])=\nu K_d[\rchi]$.
\end{proposition} 
\begin{proof} 
1) $\Rightarrow$. Assume that, for some integer $n\in \{d+1,\dots,d'\}$, 
$\nu(\rchi^n-K_{n-1}[\rchi])$ has a greatest element $\nu(f)$. 
If $\nu(f)$ is not in $\nu K_d[\rchi]$, then $\nu K_{d'}[\rchi]\neq \nu K_d[\rchi]$, and, 
by Remark \ref{extimptim} 1), the 
extension is not immediate. Suppose that $\nu(f)\in \nu K_d[\rchi]$, say $\nu(f)=\gamma$. 
For every $g\in K_d[\rchi]$ we have $\nu(f-g)\leq \nu(f)$, hence $f_{\gamma,\nu}\neq 
g_{\gamma,\nu}$. It follows that $(K_{d'}[\rchi])_{\gamma,\nu}\neq (K_d[\rchi])_{\gamma,\nu}$, so the 
extension is not immediate. \\ 
\indent $\Leftarrow$. Assume that $(K_{d'}[\rchi]|K_d[\rchi],\nu)$ is not immediate, and let $n$ be 
the smallest integer such that, for some polynomial $f$ of degree $n$, 
either $\nu(f)\notin \nu K_d[\rchi]$ or, for every $g\in K_d[\rchi]$, $f_{\gamma,\nu}\neq 
g_{\gamma,\nu}$, where $\gamma=\nu(f)$. Note that, by dividing $f$ by an element of $K$, we can assume that  
$f$ is a monic polynomial of degree $n$. \\ 
\indent First assume that $\gamma \notin \nu K_d[\rchi]$, and let 
$g$ be a monic polynomial of degree $n$. Then deg$(f-g)<n$, hence, by minimality of $n$, 
$\nu(f-g)\in \nu K_d[\rchi]$. So 
$\nu(f-g)\neq \nu(f)$. It follows: $\nu(g)=\min (\nu(g-f),\nu(f))\leq \nu (f)$, 
which proves 
that $\nu(f)$ is the greatest element of $\nu(\rchi^n-K_{n-1}[\rchi])$. \\ 
\indent Assume that $\gamma=\nu(f)\in \nu K_d[\rchi]$, and, for every $g\in K_d[\rchi]$, $f_{\gamma,\nu}\neq 
g_{\gamma,\nu}$. Let $g$ be a monic polynomial of degree $n$. If 
$\nu(g-f)\neq \gamma$, then $\nu(g)=\min (\nu(g-f),\nu(f))\leq \nu(f)$. Now, assume that 
$\nu(g-f)=\gamma$. We have deg$(f-g)<n$, hence, by minimality of $n$, $(g-f)_{\gamma,\nu}=h_{\gamma,\nu}$ 
for some $h\in K_d[\rchi]$. Therefore we have $f_{\gamma,\nu}\neq 
(g-f)_{\gamma,\nu}$. So $\nu(g)=\nu(g-f+f)=\min (\nu(g-f),\nu(f))\leq \nu(f)$. Consequently, $\nu(f)$ is 
the greatest element of $\nu(\rchi^n-K_{n-1}[\rchi])$. \\ 
\indent 2) If $(K_{d'}[\rchi]|K_d[\rchi],\nu)$ is dense, then it is immediate. Hence for every integer $n\in \{d+1,\dots,d'\}$ 
we have: 
$\nu(\rchi^n-K_{n-1}[\rchi])\subseteq \nu K_d[\rchi]$. Furthermore: $\nu(\rchi^{d+1}-K_d[\rchi])=\nu K_d[\rchi]$. Now, 
for every integer $n\in \{d+2,\dots,d'\}$ we have $\nu(\rchi^n-K_{n-1}[\rchi])\supseteq \nu(\rchi^n-\rchi^{n-(d+1)}K_d[\rchi])=
(n-(d+1))\nu(\rchi)+\nu(\rchi^{d+1}-K_d[\rchi])=\nu K_d[\rchi]$ (since $\nu K_d[\rchi]$ is a subgroup). 
It follows: $\nu(\rchi^n-K_{n-1}[\rchi])=\nu K_d[\rchi]$. \\ 
\indent Conversely, assume that for every integer $n\in \{d+1,\dots,d'\}$ we have $\nu(\rchi^n-K_{n-1}[\rchi])=\nu K_d[\rchi]$. 
Then in 
particular it has no greatest element. So by a) $(K_{d'}[\rchi]|K_d[\rchi],\nu)$ is immediate. 
Let $f\in K_{d'}[\rchi]$, $n$ be its degree and $x_n$ be the coefficient of $\rchi^n$ in $f$. Without loss of 
generality we can assume that the integer $n$ belongs to $\{d+1,\dots, d'\}$. We show that for every $\gamma \in \nu K_d[\rchi]$ 
there is $g\in K_d[\rchi]$ such that 
$\nu(f-g)>\gamma$. For every $g_{n-1}\in K_{n-1}[\rchi]$ we have $f/x_n-g/x_n\in \rchi^n-K_{n-1}[\rchi]$. Hence 
since $\nu(\rchi^n-K_{n-1}[\rchi])=\nu K_d[\rchi]$ which is a subgroup, 
we can choose $g_{n-1}\in K_{n-1}[\rchi]$ such that $\nu((f/x_n)-(g_{n-1}/x_n))>\gamma-\nu(x_n)$. Hence 
$\nu(f-g_{n-1})>\gamma$. In the same way, for $d\leq j\leq n-2$ we get $g_j\in K_j[\rchi]$ such that 
$\nu(g_{j+1}-g_j)>\gamma$. Therefore $\nu(f-g_d)=\nu(f-g_{n-1}+\cdots +g_{d+1}-g_d)\geq 
\min (\nu(f-g_{n-1}),\dots,  \nu(g_{d+1}-g_d))> \gamma$. So every 
element of $K_{d'}[\rchi]$ is limit over $K_d[\rchi]$. This implies that $(K_{d'}[\rchi]|K_d[\rchi],\nu)$ is dense. 
\end{proof}
\indent We turn to vs defectless extensions, and we start with properties of separated sequences. 
\begin{lemma}\label{BauranglaisS1} (\cite[p.\ 676]{B2}) 
Let $(l_1,\dots,l_n)$ be a separated sequence of
elements of $L$, $y$ in $L\backslash \{0\}$ and $k_1,\dots,k_n$ in $K\backslash \{0\}$. 
The following holds. \\
Every subsequence of $(l_1,\dots,l_n)$ is separated. \\
The sequence $(k_1l_1,\dots,k_nl_n)$ is separated. \\
If $\nu$ is multiplicative on $L$, then the sequence $(y l_1,\dots,y l_n)$ is separated. 
\end{lemma}
\begin{lemma}\label{BauranglaisS4} (\cite{B1}, \cite[(S4), p.\ 676]{B2}) 
Let $(l_1,\dots,l_n)$ be a sequence of elements of $L$ such that: 
$\forall i$, 
$1 \leq i \leq n$, $\nu (l_i)=0$. Then $(l_1,\dots,l_n)$ is separated if, and only if, 
$(l_1)_{\nu}, \dots ,(l_n)_{\nu}$ are linearly 
independent over $K_{\nu}$. This can be generalized in the following way. If 
$\nu(l_1)=\cdots=\nu(l_n)=g$, then $(l_1,\dots,l_n)$ is separated if, and only if, for every 
$x_1,\dots,x_n$ in $\{x\in K\mid \nu(x)=0\}\cup\{0\}$, either 
$\nu(x_1l_1+\cdots+x_nl_n)=g$ or $x_1=\cdots=x_n=0$. \end{lemma}
\begin{proposition}\label{BauranglaisS5}
Let $l_{i 1}, \dots ,l_{i n_i}$, 
$1 \leq i \leq p$, be sequences which satisfy:
$$\forall i, 1 \leq i \leq p, \forall j, 1 \leq j \leq n_i,  
\nu (l_{i j}) = \nu (l_{i 1}) < \infty$$
and the $ \nu (l_{i 1})$ are pairwise non-congruent modulo $\nu K$.  
The following assertions are equivalent. \\
The sequence $l_{1 1}, \dots ,l_{1 n_1},l_{2 1}, \dots ,
l_{2 n_2},\dots,l_{p 1}, \dots ,l_{p n_p}$ is separated. \\
For every $i$ in $\{1,\dots,p\}$, $l_{i1},l_{i2}, \dots,l_{in_i}$ 
is separated. \\ 
If $\nu$ is multiplicative on $L$ then this condition is equivalent to: \\
for every $i$ in $\{1,\dots,p\}$, $1,(l_{i2} l_{i1}^{-1})_{\nu}, \dots ,
(l_{in_i} l_{i1}^{-1})_{\nu}$ are linearly independent over $K_{\nu}$.\end{proposition}
\begin{proof}
Assume that the sequence is separated. Then by Lemma \ref{BauranglaisS1}, 
for $1 \leq i \leq p$, the sequence $l_{i1},l_{i2}, \dots,l_{in_i}$ is separated. \\
\indent 
Conversely, let $x_{1 1}, \dots ,x_{1 n_1},x_{2 1}, \dots ,
x_{2 n_2},\dots,x_{p 1}, \dots ,x_{p n_p}$ in $K$. 
For $1 \leq i \leq p$, set 
$y_i = x_{i1}l_{i1}+\cdots+x_{in_i}l_{in_i}$. 
Since 
$l_{i 1}, \dots ,l_{i n_i}$ is separated, we have: 
$\nu (y_i)= \min \{\nu (x_{ij})+\nu (l_{ij}) \mid 1 \leq j \leq n_i \}$. Therefore, the 
$\nu (y_i)$'s are pairwise non-congruent modulo $\nu K$. In 
particular, they are pairwise distinct, and 
$\nu (y_1+\cdots +y_p)=\min \{\nu (y_i) \mid 1 \leq i \leq p \}$. This proves that 
the sequence $l_{1 1}, \dots ,l_{1 n_1},l_{2 1}, \dots ,
l_{2 n_2},\dots,l_{p 1}, \dots ,l_{p n_p}$ is separated. \\ 
\indent If $\nu$ is multiplicative on $L$, then 
$$l_{i1},l_{i2}, \dots,l_{in_i}\mbox{ is separated 
if, and only if, }1,(l_{i2} l_{i1}^{-1}), \dots ,
(l_{in_i} l_{i1}^{-1}) \mbox{ is separated.}$$ 
By Lemma \ref{BauranglaisS4}, this in turn is equivalent to 
$1,(l_{i2} l_{i1}^{-1})_{\nu}, \dots ,(l_{in_i} l_{i1}^{-1})_{\nu}$ are linearly independent  
over $K_{\nu}$.\end{proof} 
\begin{lemma}\label{ajoutelt} Assume that 
$(L|K,\nu)$ is vs defectless. Let $M$ be a finite $K$-submodule 
of $L$ and $l\in L\backslash M$. Then every separated basis of $M$ extends to a separated 
basis of the $K$-submodule generated by $M$ and $l$.
\end{lemma}
\begin{proof} The proof is left to the reader (see for example \cite{L18}). \end{proof}
\begin{proposition}\label{prop211} Let $d\in \NN\cup \{\infty\}$ such that 
$0< d$ and if $\rchi$ is algebraic over $K$, then $d< [K[\rchi]\!:\! K]$. Then $(K_d[\rchi]|K,\nu)$ is 
vs defectless if, and only if, the $K$-module $K_d[\rchi]$ admits a separated basis. Furthermore, we can 
assume that every polynomial of this basis is monic and that the mapping $f\mapsto {\rm deg}(f)$, 
from this basis onto $\NN$, is one-to-one. \\
Assume that $\rchi$ is transcendental over $K$, and that $\nu$ is multiplicative on $K(\rchi)$.
Then $(K(\rchi)|K,\nu)$ is 
vs defectless if, and only if, the $K$-module $K[\rchi]$ admits a separated basis. 
($\Rightarrow$ holds even if $\nu$ is not multiplicative). 
\end{proposition} 
\begin{proof} Assume that  $(K_d[\rchi]|K,\nu)$ is vs defectless. By Lemma \ref{ajoutelt}, 
the separated basis $1$ of $K$ extends to a separated basis of the module generated by $1$ and $\rchi$. 
Necessarily, the second element of this basis has degree $1$. Let $n\geq 1$ and assume 
that the $K$-module $K_n[\rchi]$ of polynomials of degree at most $n$ has a separated basis of $(n+1)$ 
elements of respective degrees $0$, $1$, $\dots$, $n$. By Lemma \ref{ajoutelt}, this separated basis can be 
completed in a separated basis of $K_{n+1}[\rchi]$, and the  degree of the new element is $n+1$. So we 
get the required separated basis by induction. By Lemma \ref{BauranglaisS1} we can assume that 
every polynomial of this basis is monic. \\ 
\indent Conversely, assume that $K_d[\rchi]$ contains a separated basis, and let $M$ be 
a finitely generated $K$-submodule. 
By Lemma \ref{D88lm5}, $M$ has a separated basis. Hence $(K(\rchi)|K,\nu)$ is vs defectless. \\
\indent We turn to the case where $\rchi$ is transcendental over $K$. It follows from the definition 
(Definition \ref{def14} 2)) that if $(K(\rchi)|K,\nu)$ is vs defectless, then so is $(K[\rchi]|K,\nu)$. 
Now we assume that  $(K[\rchi]|K,\nu)$ is vs defectless 
and that $\nu$ is multiplicative on $K(\rchi)$. Let $M$ 
be a finitely 
generated submodule of $K(\rchi)$. Then there is a polynomial $f\neq 0$ such that 
$f \!\cdot\!M\subseteq K[\rchi]$. 
We take a separated basis of $f \!\cdot\!M$, 
and we divide all its elements by 
$f$ so that, since $\nu$ is multiplicative, by Lemma \ref{BauranglaisS1}, 
we get a separated basis of $M$. 
\end{proof}
\indent We state a refinement of Theorem \ref{caracsep}, which characterizes vs defectless 
extensions by means of initial segments. This proposition completes Proposition \ref{prop111}. 
\begin{proposition}\label{prop214} Let $d\in \NN\cup \{\infty\}$ such that 
$0< d$ and if $\rchi$ is algebraic over $K$, then $d< [K[\rchi]\!:\! K]$. 
The extension $(K_d[\rchi]|K,\nu)$ is 
vs defectless if, and only if, for every integer $n$, $1\leq n \leq d$, $\nu(\rchi^n-K_{n-1}[\rchi])$ has a maximal 
element. \\
Assume that $\rchi$ is transcendental over $K$, and that $\nu$ is multiplicative on $K(\rchi)$. 
Then $(K(\rchi)|K,\nu)$ is 
vs defectless if, and only if, for every $n\in \NN\backslash\{0\}$, $\nu(\rchi^n-K_{n-1}[\rchi])$ has a maximal element. 
\end{proposition} 
\begin{proof} In both equivalences, $\Rightarrow$ follows from Theorem \ref{caracsep}. In order to 
prove the converse, we construct by induction a separated basis such that the mapping $f\mapsto {\rm deg}(f)$ is 
one-to-one. Then, by Proposition \ref{prop211}, $(K_d[\rchi]|K,\nu)$ is vs defectless. The case 
where $\rchi$ is transcendental also follows from Proposition \ref{prop211}. 
Trivially, $1$ is a separated basis of $K_0[\rchi]=K$. Assume that we have a separated 
basis $(f_0,\dots,f_{n-1})$ of $K_{n-1}[\rchi]$. Let $f_n$ be a monic polynomial such that 
$\nu(f_n)=\max (\nu(\rchi^n-K_{n-1}[\rchi]))$. Since the degree of $f_n$ is $n$, 
$(f_0,\dots,f_{n-1},f_n)$ is a basis of $K_n[\rchi]$. Let $f:=x_nf_n+\cdots+x_0f_0$ in $K_n[\rchi]$. 
If $x_n=0$, then by induction hypothesis $\displaystyle{\nu(f)=\min_{0\leq i\leq n-1} 
\nu(x_if_i)=\min_{0\leq i\leq n} \nu(x_if_i)}$. Now we assume: $x_n\neq 0$. Since 
$\nu(f_n)$ is maximal, we have $\displaystyle{\nu\left(\frac{f}{x_n}\right)\leq \nu(f_n)}$. 
If $\nu\left(\frac{f}{x_n}-f_n\right)<\nu(f_n)$, 
then $\displaystyle{\nu\left(\frac{f}{x_n}-f_n\right)= \nu\left(\frac{f}{x_n}\right) }$. So: 
$$\nu(f)=\nu(x_n)+\nu\left(\frac{f}{x_n}\right)=\nu(x_n)+\nu\left(\frac{f}{x_n}-f_n\right)=
\nu(x_n)+\nu\left(\frac{x_{n-1}f_{n-1}}{x_n}+\cdots +\frac{x_0f_0}{x_n}\right)=$$
$$=\nu(x_n)+\min_{0\leq i\leq n-1} \nu\left(\frac{x_if_i}{x_n}\right)=
\min_{0\leq i\leq n-1} \nu(x_if_i)=\min_{0\leq i\leq n} \nu(x_if_i).$$
If $\displaystyle{\nu\left(\frac{f}{x_n}-f_n\right)\geq \nu(f_n)}$, 
then $\displaystyle{\min_{0\leq i\leq n-1} \nu(x_if_i)
\geq \nu(x_nf_n)\geq\nu(f)\geq \min_{0\leq i\leq n} \nu(x_if_i)}$. \\ 
Therefore: $\displaystyle{\nu(f)=\min_{0\leq i\leq n} \nu(x_if_i)}$. 
\end{proof}
\subsection{Graded algebra associated to a valuation}\label{subsection23}
In the proofs of Proposition \ref{rk322} and Theorem \ref{propdiv} we will introduce the graded algebra 
associated to a valuation.
Now, we review some observations. \\
\indent 
Let $(K,\nu)$ be a valued field. Recall that, for every $\gamma\in \nu K$, $K_{\gamma,\nu}$ denotes the 
$K_{\nu}$-module $\{x\in K\mid \nu(x)\geq \gamma\}/\{x\in K\mid \nu(x)>\gamma\}$. Now, let 
$G_{\nu}(K)$ be the graded algebra  
$\displaystyle{G_{\nu}(K)=\bigoplus_{\gamma\in \nu K} K_{\gamma,\nu}}$. \\
\indent 
In the case where $K$ is the valued field $k((\Gamma))$ of generalized formal power series with 
coefficients in a field $k$ and exponents in a linearly ordered abelian group $\Gamma$: 
$$k((\Gamma))=
\left\{ \sum_{\gamma\in \Lambda}x_{\gamma}X^{\gamma}\mid \Lambda \mbox{ is a well-ordered subset of } 
\Gamma, \mbox{ and } \forall \gamma\in \Lambda\; x_{\gamma}\in k\right\},$$ 
then $G_{\nu}(K)$ is isomorphic to the ring of generalized polynomials $k[\Gamma]$, the subring of 
all elements of $k((\Gamma))$ with finite support. 
Note that if $k_1|k_0$ is an extension of fields and $\Gamma_1|\Gamma_0$ is an extension of groups, then 
$k_1[\Gamma_1]$ is a $k_0[\Gamma_0]$-module. Its dimension is finite if, and only if, both of 
$[k_1\!:\!k_0]$ 
and $(\Gamma_1:\Gamma_0)$ are finite. If this holds, then this dimension is $[k_1\!:\!k_0](\Gamma_1:\Gamma_0)$. \\
\indent More generally, let $\gamma\in \nu K$ and pick an element $x_0$ of valuation $\gamma$. Then 
the mapping $x_{\nu(x),\nu}\mapsto (xx_0^{-1})_{\nu}$ induces a $K_{\nu}$-module isomorphism 
between $\{x\in K\mid \nu(x)\geq \gamma\}/\{x\in K\mid \nu(x)>\gamma\}$ and $K_{\nu}$. Hence, 
the $K_{\nu}$-module $G_{\nu}(K)$ is isomorphic to the $K_{\nu}$-module $K_{\nu}[\nu K]$ of 
polynomials with coefficients in $K_{\nu}$ and exponents in $\nu K$.  
If $K$ contains a lifting of $\nu K$, then we can assume that 
these graded algebras are isomorphic. In particular, if $\nu K=\ZZ$, then they are isomorphic. 
If $(K',\nu')$ is an $\aleph_1$-saturated elementary extension of $(K,\nu)$, then it contains a 
lifting of its value group (see \cite{K75}). Hence $G_{\nu'}(K')$ is isomorphic to the ring of 
polynomials $K_{\nu'}'[\nu' K']$. Therefore every graded algebra $G_{\nu}(K)$ embeds in a ring of 
polynomials. 
If $(K,\nu)$ contains a lifting $K_0$ of its residue field and a lifting $\Gamma$ of $\nu K$, then 
it contains the algebra $K_0[\Gamma]$, which is isomorphic to $G_{\nu}(K)$. Now, if $(K,\nu)$ is 
henselian and char$(K_{\nu})=0$, then we know that it admits a lifting of $K_{\nu}$. It 
follows that every valued field $(K,\nu)$ of residue characteristic $0$ admits an 
extension $(K',\nu')$ which contains a subalgebra which is isomorphic to $G_{\nu}(K)$. Furthermore, 
$(K',\nu')$ embeds in the power series field $K_0'((\nu' K'))$ equipped with the canonical valuation. \\ 
\indent 
For every $x\in K$, let $in_{\nu}(x):=x_{\nu(x),\nu}$ be the image of $x$ in $K_{\nu(x),\nu}$, 
which is also its image in $G_{\nu}(K)$. In the case of a subfield of a power series field, we have 
$\displaystyle{in_{\nu}\left(\sum_{\gamma\in \Lambda}x_{\gamma}X^{\gamma} \right)=
x_{\gamma_0}X^{\gamma_0}}$, 
where $\gamma_0$ is the smallest element of the {\it support} of the series (i.e.\ the well ordered subset 
$\Lambda$ of $\nu K$ such that $x_{\gamma_0}\neq 0$). 
In general, for every $x$, $y$ in $K$, we have $in_{\nu}(x)in_{\nu}(y)=
in_{\nu}(xy)$. Assume that $\nu(x)=\nu(y)$. If $in_{\nu}(x)=-in_{\nu}(y)$, then $\nu(x+y)>\nu(x)$ and 
$\nu(x+y)$ can be arbitrarily large, 
so we have $in_{\nu}(x+y)=0$. 
Otherwise, $in_{\nu}(x+y)=in_{\nu}(x)+in_{\nu}(y)$. \\
\indent An element of $G_{\nu}(K)$ is called {\it homogeneous} if it belongs to 
$\displaystyle{\bigcup_{\gamma\in\nu K} K_{\gamma,\nu} }$. In the case of a polynomial ring, this is 
equivalent to being a monomial. One can see that the invertible elements of $G_{\nu}(K)$ are the 
homogeneous ones. \\
\indent For further purposes, if $M$ is a $K$-submodule of $L$, we denote by 
$G_{\nu}(M)$ the additive group  
$\displaystyle{\bigoplus_{\gamma\in \nu M} M_{\gamma,\nu}}$.  \\[2mm]
\indent By Remark \ref{extimptim}, an extension of valued fields $(L|K,\nu)$ is immediate if, 
and only if, $G_{\nu}(L)=G_{\nu}(K)$. \\
\indent Now we turn to the vs defectless case. 
\begin{proposition}\label{fact220} 
Let $(L|K,\nu)$ be an extension of valued fields, and $l_1,\dots,l_n$ in $L$. 
The family  $(l_1,\dots,l_n)$ is separated over $(K,\nu)$ if, and only if, 
$in_{\nu}(l_1),\dots,in_{\nu}(l_n)$ are li\-ne\-ar\-ly independent over $G_{\nu}(K)$. \\
 If $(l_1,\dots,l_n)$ is a maximal separated family, 
then $(in_{\nu}(l_1),\dots,in_{\nu}(l_n))$ is a basis of $G_{\nu}(L)$. \\
Assume that $L|K$ is finite. Then $(L|K,\nu)$ is vs defectless if, and only if, $G_{\nu}(L)$ is a 
$G_{\nu}(K)$-module of dimension $[L\!:\! K]$.  \end{proposition}
\begin{proof} Recall that
the family  $(l_1,\dots,l_n)$ is separated over $(K,\nu)$ if, and only if, 
for every $x_1,\dots,x_n$ in $K$, $\nu(x_1l_1+\cdots+x_nl_n)=\min (\nu(x_1l_1),\dots,\nu(x_nl_n))$. 
Now, this equivalent to saying that for every $x_1,\dots,x_n$ in $K$ with 
$\nu(x_1l_1)=\cdots =\nu(x_nl_n)$, we have $\nu(x_1l_1+\cdots+x_nl_n)=\nu(x_1l_1)$. This last equality is 
equivalent to $in_{\nu}(x_1)in_{\nu}(l_1)+\cdots+in_{\nu}(x_n)in_{\nu}(l_n)\neq 0$. 
So, if $in_{\nu}(l_1),\dots,in_{\nu}(l_n)$ are linearly independent in the $G_{\nu}(K)$-module $G_{\nu}(L)$, 
then the family $(l_1,\dots,l_n)$ is separated. Now, assume that for every $x_1,\dots,x_n$ in $K$ with 
$\nu(x_1l_1)=\cdots =\nu(x_nl_n)$, we have $\nu(x_1l_1+\cdots +x_nl_n)=\nu(x_1l_1)$. Let $y_1,\dots,y_n$ in 
$G_{\nu}(K)$. Every $y_j$ can be written as a finite sum of homogeneous elements: 
$y_j=in_{\nu}(x_{j,1})+\cdots+in_{\nu}(x_{j,i_j})$. It follows that $y_1 in_{\nu}(l_1)+\cdots+y_n in_{\nu}(l_n)$ 
can be written as a sum of $in_{\nu}(x_{1,k_1})in_{\nu}(l_1)+\cdots +in_{\nu}(x_{n,k_n})in_{\nu}(l_n)$'s, 
where the non-zero $in_{\nu}(x_{j,k_j})in_{\nu}(l_j)$ have the same valuation. Therefore,  
$y_1in_{\nu}(l_1)+\cdots+y_nin_{\nu}(l_n)\neq 0$. 
Consequently, the family  $(l_1,\dots,l_n)$ is separated over $(K,\nu)$ if, and only if, 
$in_{\nu}(l_1),\dots,in_{\nu}(l_n)$ are linearly independent over $G_{\nu}(K)$. 
Furthermore, if $(l_1,\dots,l_n)$ is a maximal separated family, 
then $(in_{\nu}(l_1),\dots,in_{\nu}(l_n))$ is a basis of $G_{\nu}(L)$. 
Now, if $[L\!:\!K]$ is finite, then the dimension of the $G_{\nu}(K)$-module $G_{\nu}(L)$ is 
$[L_{\nu}\! :\! K_{\nu}]\!\cdot \!(\nu L\! :\! \nu K)$. Hence, by Theorem \ref{relatdegsep}, 
$(L|K,\nu)$ is vs defectless if, and only if, $G_{\nu}(L)$ is a $G_{\nu}(K)$-module of dimension 
$[L\!:\! K]$.  \end{proof}
\indent The following lemma shows that if, for $l\in L$, $in_{\nu}(l)$ satisfies a relation of 
algebraic dependence over $G_{\nu}(K)$, then we can define its irreducible polynomial. 
\begin{lemma}\label{irrpol} Let $(L|K,\nu)$ be an extension of valued fields, and $l\in L$. 
Assume that $in_{\nu}(l)$ satisfies a relation of 
algebraic dependence over $G_{\nu}(K)$, and let $n$ be the smallest degree such that 
$in_{\nu}(l)$ satisfies a relation of algebraic dependence of degree $n$. 
Then, $in_{\nu}(l)$ satisfies a relation of the form $in_{\nu}(l)^n+in_{\nu}(x_{n-1})
in_{\nu}(l)^{n-1}+\cdots+in_{\nu}(x_0)=0$, where $x_0,\dots,x_{n-1}$ belong to $K$ and 
$\nu(x_0)=\cdots=\nu(x_{n-1}l^{n-1})=\nu(l^n)$. 
\end{lemma}
\begin{proof} See for example \cite{HOS07}. 
\end{proof}
%
\section{Key polynomials}\label{section3}
\indent In this section, $\rchi$ is algebraic or transcendental over the field $K$ and 
$\nu$ is a $K$-module valuation on $K(\rchi)$.  
\subsection{Definitions}\label{subsection31}
\begin{notation} Let $\Phi$ be a monic polynomial of degree $d\geq 1$. 
For $f$, $g$ in $K_{d-1}[\rchi]$, we will denote by $q_{\Phi}(f,g)$ and $r_{\Phi}(f,g)$ 
respectively (in short $q(f,g)$ and $r(f,g)$) the quotient and the remainder of the 
euclidean division of $fg$ by $\Phi$. In other words, $q(f,g)$ and $r(f,g)$ belong to 
$K_{d-1}[\rchi]$ and $fg=\Phi\!\cdot\!q(f,g)+r(f,g)$. 
\end{notation} 
\begin{definitions}\label{def32} 
Let $\Phi$ be a monic polynomial of degree $d\geq 1$. \\
We say that $\Phi$ is a {\it weak key polynomial for} $\nu$ (or a {\it w key polynomial for} $\nu$) 
 if, for every $f$, $g$ in 
$K_{d-1}[\rchi]$ with $\mbox{deg}(f)+\mbox{deg}(g)<[K[\rchi]\!:\!K]$, we have $\nu(fg)=\nu(r(f,g))$. \\
We say that $\Phi$ is a {\it key polynomial for} $\nu$ if, for every $f$, $g$ in $K_{d-1}[\rchi]$ 
with $\mbox{deg}(f)+\mbox{deg}(g)<[K[\rchi]\!:\!K]$, we have $\nu(fg)=\nu(r(f,g))<\nu(q(f,g)\!\cdot\!\Phi)$. \\
Let $d$ be a positive integer. \\
We say that $d$ is a {\it key degree} of $(K(\rchi)|K,\nu)$ if 
there exists a key polynomial of degree $d$. \\
Assume that $d$ is a key degree. If $\nu(\rchi^d-K_{d-1}[\rchi])$ has no maximal 
element, then we say that $d$ is an {\it immediate} key degree. Otherwise, 
we say that $d$ is an {\it independent} key degree. If this maximum does not belong to 
$\nu K_{d-1}[\rchi]$, then we say that $d$ is a {\it valuational} key degree. 
If this maximum belongs to $\nu K_{d-1}[\rchi]$, then we say that $d$ is a {\it residual} key degree. 
\end{definitions}
\indent Let $d$ be a key degree. Then, by Proposition \ref{prop111}, 
$d$ is an immediate key degree if, and only if, the extension 
$(K_d[\rchi]|K_{d-1}[\rchi],\nu)$ is immediate and by Proposition \ref{prop214}, 
$d$ is a vs defectless key degree if, and only if, the extension 
$(K_d[\rchi]|K_{d-1}[\rchi],\nu)$ is vs defectless. However, we prefer to say that these degrees 
are independent because this word is easier to pronounce than vs defectless. We will see after 
Lemma \ref{avpropdiva} that this makes sense. 
\begin{proposition}\label{rks33} 
1) The integer $1$ is a key degree. Furthermore, 
every monic polynomial of degree $1$ is a key polynomial.\\ 
2) Every w key polynomial is irreducible. \\
3) If $\nu$ is partially multiplicative and $\Phi$ is a w key polynomial (resp.\ a key polynomial) 
for $\nu$ of degree $d$, then, for every pm valuation $\nu'$ 
such that the restriction of $\nu'$ to $K_{d-1}[\rchi]$ is equal to the restriction of $\nu$ 
and $\nu'(\Phi)\geq \nu(\Phi)$, $\Phi$ is a w key polynomial (resp.\ a key polynomial) for $\nu'$. \\
4) Assume that $\nu$ is partially multiplicative. If 
$\Phi$ is a w key polynomial, then $\nu(\Phi)\geq \{\nu(r(f,g))-\nu(q(f,g))\mid 
f\in K_{d-1}[\rchi],\; g\in K_{d-1}[\rchi],\; {\rm deg}(f) +{\rm deg}(g)<[K[\rchi]\!:\!K]\}$. 
If $\Phi$ is a monic polynomial, then $\Phi$ is a key polynomial if, and only if, 
$\nu(\Phi)> \{\nu(r(f,g))-\nu(q(f,g))\mid 
f\in K_{d-1}[\rchi],\; g\in K_{d-1}[\rchi],\; {\rm deg}(f) +{\rm deg}(g)<[K[\rchi]\!:\!K]\}$.
\end{proposition}
\begin{proof} 1) and 3) are trivial. \\
\indent 2) Assume that there exist two polynomials $f$, $g$ in $K_{d-1}[\rchi]$ such that 
$fg=\Phi$, then $r(f,g)=0$, and $\nu(r(f,g))=\infty> \nu(fg)$. 
Hence $\Phi$ is not a w key polynomial for $\nu$. \\
\indent 4) Clearly, if $\Phi$ is a w key polynomial (resp.\ a key polynomial), then 
$\nu(\Phi)\geq \{\nu(r(f,g))-\nu(q(f,g))\mid 
f\in K_{d-1}[\rchi],\; g\in K_{d-1}[\rchi],\; {\rm deg}(f) +{\rm deg}(g)<[K[\rchi]\!:\!K]\}$ 
(resp.\ $\nu(\Phi)> \{\nu(r(f,g))-\nu(q(f,g))\mid 
f\in K_{d-1}[\rchi],\; g\in K_{d-1}[\rchi],\; {\rm deg}(f) +{\rm deg}(g)<[K[\rchi]\!:\!K]\}$).
Now, if $\nu(\Phi)>\nu(r(f,g))-\nu(q(f,g))$, then $\nu(q(f,g)\Phi)>\nu(r(f,g))=\nu(fg-q(f,g)\Phi)$.  
Hence $\nu(r(f,g))=\nu(fg)$. 
\end{proof}
\indent The following lemma explains the distinction that we make between the valuational and the 
residual key degrees.  
\begin{lemma}\label{avpropdiva} Let $d$ be a positive integer and $\Phi$ be a monic polynomial of degree $d$. 
Then $\nu(\Phi)$ is the maximum of $\nu(\rchi^d-K_{d-1}[\rchi])$ if, and only if, 
either $\nu(\Phi)\notin \nu K_{d-1}[\rchi]$, or $\Phi_{\nu(\Phi),\nu}\notin (K_{d-1}[\rchi])_{\nu(\Phi),\nu}$. 
\end{lemma}
\begin{proof} Assume that $\nu(\Phi)$ is the maximum of $\nu(\rchi^d-K_{d-1}[\rchi])$ and 
$\nu(\Phi)\in \nu K_{d-1}[\rchi]$. Then, 
for every $f\in K_{d-1}[\rchi]$ such that 
$\nu(f)=\nu(\Phi)$ we have $\nu(\Phi-f)=\nu(\Phi)=\nu(f)$. Hence $(\Phi)_{\nu(\Phi),\nu}\neq 
f_{\nu(\Phi),\nu}$. It follows that 
$\Phi_{\nu(\Phi),\nu}\notin (K_{d-1}[\rchi])_{\nu(\Phi),\nu}$. Conversely, assume that 
$\nu(\Phi)\notin \nu K_{d-1}[\rchi]$, or $\Phi_{\nu(\Phi),\nu}\notin (K_{d-1}[\rchi])_{\nu(\Phi),\nu}$. Then 
for every $g\in K_{d-1}[\rchi]$ we have $\nu(\Phi-g)=\min (\nu(\Phi),\nu(g))\leq\nu(\Phi)$. 
Hence  $\nu(\Phi)$ is the maximum of $\nu(\rchi^d-K_{d-1}[\rchi])$. 
\end{proof}
\indent It follows that if $(K_d[\rchi]|K_{d-1}[\rchi],\nu)$ is vs defectless, 
then any monic polynomial $\Phi$, such that $\nu(\Phi)$ is the maximum of $\nu(\rchi^d-K_{d-1}[\rchi])$, 
is valuation independent over $K_{d-1}[\rchi]$ (i.e.\ for every $g\in K_{d-1}[\rchi]$ we have 
$\nu(\Phi+g)=\min (\nu(\Phi),\nu(g))$. So the word independent key degree makes sense. \\[2mm]
\indent 
Let $\Phi$ be a monic irreducible polynomial of $K[X]$ of degree $d\geq 1$ ($K[X]$, the ring of formal 
polynomials with coefficients in $K$). Then 
$K[X]/(\Phi)$ is a field, such that the canonical epimorphism $\rho:\; K[X]\rightarrow 
K[X]/(\Phi)$ is an isomorphism from the $K$-module $K_{d-1}[X]$ onto the $K$-module 
$K[X]/(\Phi)$. Now, for $f,\;g$ in $K_{d-1}[X]$, 
we set $f\ast g:=r(f,g)$. Then $\rho(f\ast g)=\rho(r(f,g))=\rho(fg)$. 
Hence $(K_{d-1}[X],+,\ast)$ is 
a field which is isomorphic to $K[X]/(\Phi)$. The same operation can be defined in 
$K[\rchi]$ whenever $d\leq [K(\rchi)\!:\! K]/2$. 
The field $(K_{d-1}[\rchi],+,\ast)$ will be denoted by $K_{\Phi}$. 
Note that if $\nu$ is a pm valuation on the field $K(\rchi)$, 
then its restriction to 
$K_{d-1}[\rchi]$ induces a valuation of the $K$-modules $K_{\Phi}$ and $K[X]/(\Phi)$. 
If $\mathcal{Y}$ is a root of $\Phi(X)$ in some algebraic extension, then the fields 
$K_{\Phi}$ and $K[\mathcal{Y}]$ are isomorphic. 
\begin{proposition}\label{lm38} 
Let $\Phi$ be an irreducible monic polynomial of degree $d$, $1\leq d \leq [K(\rchi)\!:\! K]/2$, 
and  assume that $\nu$ is a pm valuation on $K(\rchi)$. Then 
$\Phi$ is a w key polynomial for $\nu$ if, and only if, the valued $K$-module $(K_{\Phi},\nu)$ 
is a valued field. 
\end{proposition}
\begin{proof} 
Let $f,\;g$ in $K_{d-1}[\rchi]$. Then $\nu(f\ast g)=\nu(r(f,g))$ and 
$\nu(fg)=\nu(f)+\nu(g)$. Hence $(K_{\Phi},\nu)$ is a valued field 
if, and only if, for every $f,\;g$ in $K_{d-1}[\rchi]$ we have: 
$\nu(r(f,g))=\nu(fg)$. This in turn is equivalent to saying that $\Phi$ is a w key polynomial. \end{proof}
\begin{corollary}\label{cor39} 
Assume that $\nu$ is a pm valuation on $K(\rchi)$ and let $\Phi$ be a w key polynomial of degree $d$, 
$1\leq d \leq [K(\rchi)\!:\! K]/2$. Then $\nu K_{d-1}[\rchi]$ is a subgroup of $\nu K(\rchi)$. 
\end{corollary}
\begin{proof} Indeed, if $\Phi$ is a w key polynomial, then by Proposition \ref{lm38} 
$\nu K_{d-1}[\rchi]=\nu K_{\Phi}$ is a subgroup of $\nu K(\rchi)$.
\end{proof}
\begin{remark}\label{rkap36} 
Let $\Phi$ be a monic polynomial of degree $d$, $1\leq d \leq [K(\rchi)\!:\! K]/2$, 
and assume that $\nu$ is a pm valuation on $K(\rchi)$. \\
\indent 1) The polynomial $\Phi$ is a key polynomial if, and only if, for every $f$, $g$ 
in $K_{d-1}[\rchi]$, $\nu(fg)=\nu(r_{\Phi}(f,g))$ and $(fg)_{\nu(fg),\nu}=(r_{\Phi}(f,g))_{\nu(fg),\nu}$. 
Indeed, if $\nu(fg)=\nu(r_{\Phi}(f,g))$, then $\nu(fg-r_{\Phi}(f,g))>\nu(fg)$ 
is equivalent to $(fg)_{\nu(fg),\nu}=(r_{\Phi}(f,g))_{\nu(fg),\nu}$. \\
\indent  2) If $\Phi$ is a key polynomial, then 
the group $G_{\nu}(K_{d-1}[\rchi])$ is a subalgebra of $G_{\nu}(K(\rchi))$. 
Indeed, if $\Phi$ is irreducible, then the group $G_{\nu}(K_{d-1}[\rchi])$ is isomorphic to 
the group $G_{\nu}(K_{\Phi})$. Since 
$(fg)_{\nu(fg),\nu}=f_{\nu(f),\nu}g_{\nu(g),\nu}$, $\Phi$ is a key polynomial if, and only if, 
for every $f$, $g$ in $K_{d-1}[\rchi]$, $(f\ast g)_{\nu(fg),\nu}=f_{\nu(f),\nu}g_{\nu(g),\nu}$.  
Therefore, the group $G_{\nu}(K_{\Phi})$ is a subalgebra of $G_{\nu}(K(\rchi))$. \\
\indent 3) It follows from 1) and Proposition \ref{lm38} that if $\Phi$ is a key polynomial, 
then $K_{d-1}[\rchi]_{\nu}$ is a subfield of $K(\rchi)_{\nu}$. 
\end{remark} 
\subsection{Pseudo-valuations}\label{subsection32b}
In \cite{V07}, $\rchi$ is transcendental over $K$ and $\nu$ is a valuation or a pseudo-valuation. 
A {\it pseudo-valuation} $\nu$ of $K[\rchi]$ is a mapping from $K[\rchi]$ onto a linearly 
ordered group $\nu (K[\rchi])$ together with an element $\infty$ which shares the properties 
of multiplicative valuations except $\nu(f)=\infty \Rightarrow f=0$. In this case the set 
$I:=\{f\in K[\rchi] \mid \nu(f)=\infty\}$ is a prime ideal of $K[\rchi]$. 
Then $\nu$ induces a pm valuation on the integral domain $K[\rchi]/I$. 
So, the case of algebraic extensions is obtained by means of a pseudo-valuation, by considering the quotient 
of $K[\rchi]$ by the ideal $I:=\{f\in K[\rchi]\; :\; \nu(f)=\infty\}$. \\
\indent Let $\Pi$ be the monic polynomial which generates $I$. Since $I$ is prime, $\Pi$ is irreducible. 
If $\Phi$ is another irreducible polynomial, then $\Phi$ and $\Pi$ are coprime. Hence there exist polynomials $a$ and $b$ 
in $K[\rchi]$ such that ${\rm deg}(a)<{\rm deg}(\Phi)$, ${\rm deg}(b)<{\rm deg}(\Pi)$ and 
$a\Pi-b\Phi=1$. Then $a\Pi=b\Phi+1$, with $\nu(a\Pi)=\infty$, hence $\nu(b\Phi)=\nu(1)=0<\nu(a\Pi)$. 
It follows that if ${\rm deg}(\Pi)<{\rm deg}(\Phi)$, then $\Phi$ is not a w key polynomial. 
Now, $\Pi$ satisfies the properties of Definitions \ref{def32}, but since the aim is to study the quotient 
$K[\rchi]/I$, we can restrict to polynomial of degrees less that ${\rm deg}(\Pi)$. \\
\indent We generalize Definitions \ref{def32} to the case where $\nu$ is a pseudo-valuation by adding the 
hypothesis that $d={\rm deg}(\Pi)<{\rm deg}(\Phi)$. In the case where $\nu$ is a pseudo-valuation, we also assume that  
$\rchi$ is transcendental. The restriction of $\nu$ to $K_{d-1}[\rchi]$ is a valuation, hence 
Proposition \ref{rks33} remains true. In Lemma \ref{avpropdiva} we can also assume that 
$\nu$ is a pseudo-valuation by adding the condition $\nu(\Phi)<\infty$. The same holds for the other properties of 
Subsection \ref{subsection31}, we only require that $d<{\rm deg}(\Pi)$. \\
\indent More generally, all properties of Section \ref{section3} remain true with pseudo-valuations. Then we assume that 
the degrees are bounded by ${\rm deg}(\Pi)-1$. In general, we replace $d<[K[\rchi]\!:\!K]$ by $d<{\rm deg}(\Pi)$. 
\subsection{Vs defectless valuations defined by key polynomials}\label{subsec33}
\indent Assume that $\Phi$ is a monic irreducible polynomial of degree $d$, and let $\gamma$ be 
an element of an extension of $\nu K(\rchi)$. \\
\indent 
For every $f:=f_0+f_1\Phi+\cdots+f_m\Phi^m$ in $K[\rchi]$ (with $f_0,\;f_1,\; \dots,\; f_m$ 
in $K_{d-1}[\rchi]$ and ${\rm deg}(f_m)+dm<[K[\rchi]\!:\! K]$), set       
$\displaystyle{\nu'(f)=\min_{0\leq i\leq m} \nu(f_i)+i\gamma}$. \\
\indent 
Assume that $\nu$ is a pm valuation on $K(\rchi)$ and that $\Phi$ is a w key polynomial 
for $\nu$ of degree $d$. We saw in Proposition \ref{rks33} 4) that the set 
$\{\nu(r(f,g))- \nu(q(f,g))\mid f\in K_{d-1}[\rchi],\; g\in K_{d-1}[\rchi],\; 
{\rm deg}(f)+{\rm deg}(g)<[K[\rchi]\!:\! K]\}$ is bounded above by $\nu(\Phi)$. 
We extend the addition of elements of $\nu K$ to the addition of Dedekind cuts in the usual way.  
%
\begin{proposition}\label{avprop32} 
Let $\Phi$ be a monic irreducible polynomial of degree $d$, $\gamma$ be 
an element of an extension of $\nu K(\rchi)$, and $\nu'$ be defined as above. \\ 
1) The application $\nu'$ is a $K$-module valuation and the family 
$(\Phi^m)$ ($dm<[K[\rchi]\!:\!K]$) is separated over $K_{d-1}[\rchi]$. \\
2) Assume that $\nu$ is a pm valuation and that $\Phi$ is a w key polynomial for $\nu$. Denote by 
$\beta$ the upper-bound of the set 
$\{\nu(r(f,g))- \nu(q(f,g))\mid f\in K_{d-1}[\rchi],\; g\in K_{d-1}[\rchi],\; 
{\rm deg}(f)+{\rm deg}(g)<[K[\rchi]\!:\! K]\}$ in a Dedekind completion of $\nu K$, 
and assume that $\gamma\geq \beta$. \\ 
a) For every $f=f_m\Phi^m+\cdots+f_1\Phi+f_0$, 
$g=g_m\Phi^m+\cdots+g_1\Phi+g_0$, with 
$f_0,\dots,f_m,g_0,\dots,g_m$ in $K_{d-1}[\rchi]$ such that ${\rm deg}(f)+{\rm deg}(g)<
[K[\rchi]\!:\! K]$ we have: 
$$\nu'(fg)=\nu'\left(\sum_{j=0}^{2m}\left(\sum_{i=0}^j r(f_i,g_{j-i}) \right)
\Phi^j\right)\leq 
\nu'\left(fg-\sum_{j=0}^{2m}\left(\sum_{i=0}^j 
r(f_i,g_{j-i}) \right)\Phi^j\right)-(\gamma-\beta).$$
(Here if $i>m$, then we let $f_i=g_i=0$.) In particular, $\nu'$ is a pm valuation.  \\
b) The polynomial $\Phi$ is a w key polynomial such that the sequence 
$(\Phi^m)$ ($md<[K(\rchi)\!:\! K]$) 
is $\nu'$-separated over $K_{d-1}[\rchi]$. Furthermore it is a key polynomial 
if, and only if, $\gamma>\beta$ or $\beta$ is not a  maximum. 
\end{proposition} 
\begin{proof}The proof is left to the reader (see for example \cite{L18}).
\end{proof}
\begin{notations}\label{notat325} The $K$-module valuation $\nu'$ defined in 
Proposition \ref{avprop32} will be denoted by $\nu_{\Phi,\gamma}$. We 
set $\nu_{\Phi}=\nu_{\Phi,\nu(\Phi)}$.  If $\Phi_1$ and $\Phi_2$ are irreducible polynomials 
such that ${\rm deg}(\Phi_1)<{\rm deg}(\Phi_2)$, then we denote by $\nu_{\Phi_1,\Phi_2}$ the 
$K$-module valuation $(\nu_{\Phi_1,\nu(\Phi_1)})_{\Phi_2,\nu(\Phi_2)}$. By induction, for 
every irreducible polynomials $\Phi_1,\dots,\Phi_n$, with  
${\rm deg}(\Phi_1)<\cdots<{\rm deg}(\Phi_n)$, we define the $K$-module 
valuation $\nu_{\Phi_1,\dots,\Phi_n}$. 
\end{notations}
\indent These valuations $\nu_{\Phi}$ generalize the augmented valuations of S.\ MacLane 
(\cite{ML36a} and \cite{ML36b}). \\[2mm] 
\indent If the degree of $\Phi$ is $1$, then for every $f$, $g$ in $K_{d-1}[\rchi]=K$, 
we have  $q(f,g)=0$. Hence the set 
$\{\nu(r(f,g))- \nu(q(f,g))\mid f\in K_{d-1}[\rchi],\; g\in K_{d-1}[\rchi],\; 
{\rm deg}(f)+{\rm deg}(g)<[K[\rchi]\!:\! K]\}$ is equal to $\{-\infty\}$ and is bounded above by any element. 
Hence Proposition \ref{avprop32} proves the fact that the valuations defined in Example 
\ref{examp217} are pm valuations. 
\begin{proposition}\label{prop32} 
Let $\Phi$ be a monic polynomial of degree $d<[K(\rchi)\!:\! K]$ 
and assume that $\nu$ is a pm valuation on $K[\rchi]$. 
Then $\Phi$ is a w key polynomial for $\nu$ if, and only if, there exists a pm 
valuation $\nu'$ of 
$K[\rchi]$, such that its restriction to $K_{d-1}[\rchi]$ is equal to the restriction of $\nu$, and 
$\nu'(\Phi)>\nu(\Phi)$. \\ 
If this holds, then for every $\gamma\geq \nu(\Phi)$ in an extension of $\nu K[\rchi]$, 
$\nu_{\Phi,\gamma}$ is a pm valuation of $K[\rchi]$ such that its restriction to 
$K_{d-1}[\rchi]$ is equal 
to the restriction of $\nu$, $\nu_{\Phi,\gamma}(\Phi)=\gamma$ and $\Phi$ is a w key polynomial for 
$\nu_{\Phi,\gamma}$ such that the sequence $(\Phi^m)$ ($md<[K(\rchi)\!:\! K]$) 
is separated over $K_{d-1}[\rchi]$. 
\end{proposition}
\begin{proof} Assume that $\nu'$ is a pm valuation of $K[\rchi]$ such that its restriction to 
$K_{d-1}[\rchi]$ is equal to the restriction of $\nu$, and $\nu'(\Phi)>\nu(\Phi)$. 
Let $f$, $g$ in $K_{d-1}[\rchi]$ such that ${\rm deg}(f)+{\rm deg}(g)<
[K[\rchi]\!:\! K]$, $q=q(f,g)$ and $r=r(f,g)$. Without loss of generality we can assume that 
$q\neq 0$. We have 
$\nu(fg)=\nu(f)+\nu(g)=\nu'(f)+\nu'(g)=\nu'(fg)$. Therefore, $\nu(q\Phi+r)=\nu'(q\Phi+r)$ is greater or equal 
to both of $\min(\nu(q\Phi),\nu(r))$ and $\min(\nu'(q\Phi),\nu(r))$, where 
$\nu(q\Phi)=\nu(q)+\nu(\Phi)<\nu'(q)+\nu'(\Phi)=\nu'(q\Phi)$. 
It follows that $\nu(fg)=\nu(r)=\nu(q\Phi)<\nu'(q\Phi)$. Now, if $\Phi$ is a 
w key polynomial for $\nu$, then we can take $\nu'$ as in 
Proposition \ref{avprop32} 2). \\
\indent The last assertion also follows from Proposition \ref{avprop32}.
\end{proof}
\indent If $\nu$ and $\nu'$ are $K$-module valuations on $K(\rchi)$, then we set 
$\nu \leq \nu'$ if for every 
$f\in K[\rchi]$ we have $\nu(f)\leq \nu'(f)$. 
\begin{remark}\label{rekav312} 
1) By the definition of $\nu_{\Phi}$, for every $K$-module valuation 
$\nu'$ such that the restrictions of $\nu'$ and $\nu_{\Phi}$ 
to $K_{d-1}[\rchi]$ are equal and $\nu'(\Phi)=\nu_{\phi}(\Phi)$, we have $\nu_{\Phi}\leq \nu'$.  \\
2) Assume that $\rchi$ is algebraic over $K$, that $\nu$ is multiplicative, 
and $\nu'(f)<\nu(f)$. If $\nu' \leq \nu$, then $\nu'(1/f)\leq \nu(1/f)=-\nu(f)<-\nu'(f)$. Hence $\nu'(1/f)
\neq -\nu'(f)$. It follows that $\nu'$ is not multiplicative. Hence we cannot improve the conclusion that 
$\nu'$ is partially multiplicative in Proposition  \ref{avprop32}. 
\end{remark}
\indent In valuation theory, we say that $\nu'$ is {\it finer} than $\nu$ if 
$\forall x\; \nu'(x)\geq 0\Rightarrow \nu(x)\geq 0$ (see \cite[p.\ 54]{R}). 
Assume that $\rchi$ is algebraic over $K$. Then $K(\rchi)=K[\rchi]$, so, if $\nu'\leq \nu$, 
then $\nu'$ is finer than $\nu$. Now, any two distinct 
extensions of a valuation to an algebraic extension are incomparable (see 
\cite[Corollaire 5, p.\ 158 ]{R}). Therefore, this also proves that if $\nu'\neq\nu$, then $\nu$ or $\nu'$ is not multiplicative. 
In the case where $\rchi$ is 
transcendental over $K$ and $\nu$, $\nu'$ are valuations such that $\nu'\leq \nu$, 
then we cannot deduce that $\nu$ and $\nu'$ are comparable in the sense of Ribenboim.
Indeed, assume that $\nu(f)=\nu'(f)=\nu'(g)<\nu(g)$. Then, $\nu'(f/g)=0>\nu(f/g)$. 
Assume that $\nu'(f)<\nu(f)=\nu(g)=\nu'(g)$. Then $\nu'(f/g)<0=\nu(f/g)$. 
\begin{proposition}\label{prop320} Assume that $\nu$ is a pm valuation on $K[\rchi]$, and let 
$\Phi$ be a non constant monic polynomial in $K[\rchi]$ 
of degree $d<[K(\rchi)\!:\! K]$. Then $\Phi$ is a w key polynomial for $\nu$ if, and only if, 
there exists a pm valuation 
$\nu'\leq \nu$ such that the restrictions of $\nu$ and $\nu'$ to $K_{d-1}[\rchi]$ are equal, 
and $\Phi$ is a w key polynomial for $\nu'$ such that the sequence $(\Phi^m)$ ($md<[K(\rchi)\!:\! K]$) 
is separated over $K_{d-1}[\rchi]$.  
Furthermore, we can take $\nu'=\nu_{\Phi}$. \end{proposition}
\begin{proof} $\Leftarrow$. If $\Phi$ is a w key polynomial for $\nu'$, then, 
by 3) of Proposition \ref{rks33}, it is a w key polynomial for $\nu$. \\
\indent $\Rightarrow$. By Proposition \ref{prop32}, $\nu_{\Phi}$ is a pm valuation such that 
$\Phi$ is a w key polynomial for $\nu_{\Phi}$ such that the sequence $(\Phi^m)$ ($md<[K(\rchi)\!:\! K]$) 
is separated over $K_{d-1}[\rchi]$. By construction, the restrictions of $\nu$ and 
$\nu_{\Phi}$ to $K_{d-1}[\rchi]$ are equal. By Remark \ref{rekav312} 1) we have $\nu_{\Phi}\leq \nu$. 
\end{proof} 
\indent Thanks to the mappings $\nu_{\Phi}$ we get useful criteria for being a w key polynomial or a key polynomial, 
as show the following propositions. 
We start with a lemma. 
\begin{lemma}\label{lmequivdefs} 
Assume that $\nu$ is a pm valuation. 
Let $d<[K[\rchi]\!:\! K]$ in $\NN\backslash\{0\}$, $d':=[K[\rchi]\!:\! K]-d$, 
$\Phi$ be a non constant monic polynomial in $K[\rchi]$ 
of degree $d$. We also assume that 
for every $f$, $g$ in $K_{d-1}[\rchi]$ with ${\rm deg}(f)+{\rm deg}(g)<[K[\rchi]\!:\! K]$ 
and every $h$ in $K_{d'-1}[\rchi]$, 
we have $\nu(fg+h\Phi)=\min(\nu(fg),\nu(h\Phi))$. Then 
$\Phi$ is a w key polynomial such that the sequence $(\Phi^m)$ ($m d<[K(\rchi)\!:\! K]$) 
is separated over $K_{d-1}[\rchi]$. In particular, $\nu=\nu_{\Phi}$. 
\end{lemma}
\begin{proof}
By setting $h=-q_{\Phi}$, we have $\nu(r_{\Phi})=\nu(fg-q_{\Phi}\Phi)=
\min (\nu(fg),\nu(q_{\Phi}\Phi))\leq \nu(fg)$. 
By setting $f'=r_{\Phi}$, $g'=1$ and $h'=q_{\Phi}$, we get $\nu(fg)=\nu(r_{\Phi}+q_{\Phi}\Phi)
=\min(\nu(r_{\Phi}),\nu(q_{\Phi}\Phi))\leq \nu(r_{\Phi})$. Thefore $\nu(fg)=\nu(r_{\Phi})$. 
Hence $\Phi$ is a w key polynomial. 
Now, let $m\in \NN\backslash\{0\}$ with $dm<[K[\rchi]\!:\! K]$, $f_0,\dots, f_m$ in $K_{d-1}[\rchi]$. 
We have:  $\nu(f_m\Phi^m+\cdots+f_0)=\min(\nu((f_m\Phi^{m-1}+
\cdots+f_1)\Phi), \nu(f_ 0))$, $\nu(f_m\Phi^{m-1}+\cdots+f_1)=\min(\nu((f_m\Phi^{m-2}+
\cdots+f_2)\Phi), \nu(f_1))$, and so on. So by induction we have 
$\nu(f_m\Phi^m+\cdots+f_0)=\min(\nu(f_m\Phi^{n-1}), \dots,\nu(f_0))$. Hence 
the family $(\Phi^m)$  ($m d<[K(\rchi)\!:\! K]$) is separated over $K_{d-1}[\rchi]$. 
\end{proof}
\begin{proposition}\label{lem33} Let $\nu$, $\nu'$ be 
pm valuations on $K(\rchi)$, and $\Phi$ be a monic polynomial of degree $d$, $1\leq d<[K(\rchi)\!:\! K]$. 
Assume that their restrictions 
to $K_{d-1}[\rchi]$ are equal, that $\nu'\leq \nu$ and $\nu'(\Phi)<\nu(\Phi)$. \\ 
1) Let $f\in K[\rchi]$ and $f=q\Phi+r$ be the euclidean division of $f$ by 
$\Phi$. Then $\nu'(f)<\nu'(r)\Leftrightarrow \nu'(f)<\nu(f)$ 
and $\nu'(f)=\nu'(r)\Leftrightarrow \nu'(f)=\nu(f)$. \\
2) $\Phi$ is a w key polynomial for $\nu'$ and a key polynomial for $\nu$. \\ 
3) $\nu'=\nu_{\Phi}'=\nu_{\Phi,\nu'(\Phi)}$. \end{proposition}
\begin{proof} 1) We have $\nu'(r)=\nu(r)$ and $\nu'(q\Phi)<\nu(q\Phi)$. Assume that 
$\nu'(r)\leq\nu'(q\Phi)$. Then $\nu(f)=\min(\nu(q\Phi),\nu(r))=\nu(r)$. Furthermore, 
$\nu(f)\geq\nu'(f)\geq \min(\nu'(q\Phi),\nu(r))\geq \nu(r)=\nu(f)$. Hence $\nu'(f)=\nu(f)=\nu(r)$. 
Assume that $\nu'(q\Phi)<\nu(r)<\nu(q\Phi)$. Hence $\nu'(f)=\min(\nu'(q\Phi),\nu'(r))=
\nu'(q\Phi)<\nu'(r)=\min(\nu(q\Phi),\nu(r))=\nu(f)$. Assume that $\nu(q\Phi)\leq\nu(r)$. 
Then $\nu'(f)=\min(\nu'(q\Phi),\nu'(r))=\nu'(q\Phi)<\nu(q\Phi)=\min(\nu(q\Phi),\nu(r))
\leq\nu(f)$, and $\nu'(f)<\nu'(r)$. \\ 
\indent 2) Let $f$, $g$ in $K_{d-1}[\rchi]$ with ${\rm deg}(f)+{\rm deg}(g)<[K[\rchi]\!:\!K]$. 
Since $\mbox{deg}(f)<\mbox{deg}(\Phi)$ and 
$\mbox{deg}(g)<\mbox{deg}(\Phi)$, we have: $\nu'(f)=\nu(f)$ and $\nu'(g)=\nu(g)$, 
therefore: $\nu'(fg)=\nu'(f)+\nu'(g)=\nu(f)+\nu(g)=\nu(fg)$. 
By 1) this implies $\nu'(fg)=\nu'(r(f,g))$ and $\Phi$ is a w key polynomial for $\nu'$. 
Since the restrictions of 
$\nu$ and $\nu'$ to $K_{d-1}[\rchi]$ are equal, it follows that $\Phi$ is also a w key polynomial 
for $\nu$.  Now, for $f$, $g$ in $K_{d-1}[\rchi]$ we have $\nu(r(f,g))=\nu'(r(f,g))\leq 
\nu'(q(f,g)\!\cdot\!\Phi)<\nu(q(f,g)\!\cdot\!\Phi)$. Hence $\Phi$ is a key polynomial 
for $\nu$. \\
\indent 3) First we prove that for every $f$, $g$ in $K_{d-1}[\rchi]$ with 
${\rm deg}(f)+{\rm deg}(g)<[K[\rchi]\!:\!K]$ and $h\in K[\rchi]$ 
such that ${\rm deg}(h)<[K[\rchi]\!:\!K]-d$ 
we have: 
$\nu'(fg+h\Phi)=\min (\nu'(fg),\nu'(h\Phi))$. If $\nu'(fg)\neq \nu'(h\Phi)$, then the result 
is trivial. Assume that $\nu'(fg)=\nu'(h\Phi)$, and let $q=q(f,g)$, $r=r(f,g)$. Since $\Phi$ 
is a w key polynomial for $\nu'$, we have $\nu'(fg)=\nu'(r)$. Hence $\nu'(fg+h\Phi)\geq \nu'(r)$. 
Since $\nu'\leq \nu$, we have $\nu'(fg+h\Phi)\leq\nu(fg+h\Phi)$. Note that $fg+h\Phi=(q+h)\Phi+r$. So 
by 1), $\nu'((q+h)\Phi+r)<\nu((q+h)\Phi+r)\Leftrightarrow \nu'((q+h)\Phi+r)<\nu'(r)$: a contradiction. 
Hence $\nu'(fg+h\Phi)=\nu(fg+h\Phi)$. 
We have also $\nu(fg)=\nu'(fg)=\nu'(h\Phi)<\nu(h\Phi)$, hence $\nu(fg+h\Phi)=\nu(fg)=\nu'(fg)=
\min (\nu'(fg),\nu'(h\Phi))$. By Lemma \ref{lmequivdefs} 
we have $\nu'=\nu_{\Phi}'$. Now, since $\nu'$ and $\nu$ coincide on 
$K_{d-1}[\rchi]$ we have $\nu_{\Phi}'=\nu_{\Phi,\nu'(\Phi)}$.
\end{proof}
\begin{remark} We deduce from Proposition \ref{lem33} that if $\nu$ is a valuation on $K(\rchi)$, then 
every pm valuation $\nu'\leq \nu$, $\nu'\neq \nu$, which coincide on $K$ with $\nu$, can be written 
as $\nu'=\nu_{\Phi,\nu'(\Phi)}$ where $\Phi$ is a monic polynomial of minimal degree such that $\nu'(\Phi)<\nu(\Phi)$. 
This completes Proposition \ref{prop32}. 
\end{remark}
\begin{proposition}\label{aplem33} Let $\nu$ be a pm valuation on $K(\rchi)$, $\Phi$ 
be a w key polynomial for $\nu$ and $\Phi'$ be a monic polynomial such that 
$d={\rm deg}(\Phi)<{\rm deg}(\Phi')=d'<[K(\rchi)\!:\! K]$. \\
1) If $\nu(\Phi')=\nu_{\Phi}(\Phi')$, then $\Phi'$ is not a key polynomial for $\nu$. \\
2) Assume that 
$\nu_{\Phi}=\nu$ on $K_{d'-1}[\rchi]$. Then, $\Phi'$ is a key polynomial for $\nu$ 
if, and only if, $\nu(\Phi')>\nu_{\Phi}(\Phi')$. 
\end{proposition}
\begin{proof} 1) Write $\Phi'$ as $\Phi'=f_m\Phi^m+\cdots +f_1\Phi+f_0$, with $f_0,\dots,f_m$ 
in $K_{d-1}[\rchi]$. By the definition of $\nu_{\Phi}$, we have $\nu(\Phi')=
\nu_{\Phi}(\Phi')=\min (f_0,f_1\Phi,\dots, f_m\Phi^m)$. Let 
$f=f_m\Phi^{m-1}+\cdots +f_1$ and $g=\Phi$. Then $q_{\Phi'}(f,g)=1$,  $r_{\Phi'}(f,g)=-f_0$, 
and $\nu(q_{\Phi'}(f,g)\Phi')=\nu(\Phi')=\nu_{\Phi}(\Phi')\leq 
\min (\nu(f_1\Phi),\dots,\nu(f_n\Phi^n))\leq \nu(fg)$. 
This proves that $\Phi'$ is not a key polynomial. \\
\indent 2) $\Leftarrow$ follows from Proposition \ref{lem33}, $\Rightarrow$ follows from 1).  
\end{proof}
\section{MacLane's key polynomials}\label{subsection32}
In \cite{ML36a} and \cite{ML36b} S.\ MacLane defined key polynomials in the case of discrete 
valuations. In these papers, the key polynomials are the $\nu$-minimal and 
$\nu$-irreducible polynomials (see definitions below). 
M.\ Vaqui\'e (\cite{V07}) generalized this definition to arbitrary valuations. In 
\cite{V07} the key polynomials are defined on the ring of formal polynomials. 
Here we extend the definition of \cite{V07} to the case of an 
algebraic extension, and we do not require $\rchi$ being transcendental. \\
\indent In this section, $\nu$ is a $K$-module valuation on $K(\rchi)$ or a 
pseudo-valuation. Recall that in the case where $\nu$ is a pseudo-valuation, we also assume that  
$\rchi$ is transcendental \\
\indent If $\nu$ is a pseudo-valuation, then the results below remain true, by assuming that $d$ is lower than 
the degree of a generator of the ideal $\{f\in K[\rchi]\; :\; \nu(f)=\infty\}$. \\[2mm]
\indent Let $\Phi$  be 
a monic polynomial of degree $d$ and $d':= [K[\rchi]\!:\! K]-d$ (if $[K(\rchi)\!:\!K]=\infty$, 
then $d'=d'-1=\infty$). We say that $\Phi$ 
is {\it $\nu$-minimal} if 
for every $f\in K_{d-1}[\rchi]$ and every $h\in K_{d'-1}[\rchi]$, 
$\nu(f-h\Phi)=\min(\nu(f),\nu(h\Phi))$. 
We say that $\Phi$ {\it $\nu$-irreducible} if for every $f$, $g$ in $K[\rchi]$ 
with ${\rm deg}(f)+{\rm deg}(g)<[K[\rchi]\!:\! K]$ such that, 
for every $h \in K_{d'-1}[\rchi]$, $\nu(f-h\Phi)=\min(\nu(f),\nu(h\Phi))$ and 
$\nu(g-h\Phi)=\min(\nu(g),\nu(h\Phi))$, we have: $\forall h\in K_{d'-1}[\rchi]\; 
\nu(fg-h\Phi)=\min(\nu(fg),\nu(h\Phi))$. \\[2mm]
\indent By setting $h:=1$ in above definition, we see that every monic polynomial, 
which is $\nu$-minimal and $\nu$-irreducible, is irreducible. 
\begin{proposition}\label{equivdefs} 
Assume that $\nu$ is a pm valuation or a pseudo-valuation. 
Let $d<[K[\rchi]\!:\! K]$ in $\NN\backslash\{0\}$, $d':=[K[\rchi]\!:\! K]-d$, 
$\Phi$ be a non constant monic polynomial in $K[\rchi]$ 
of degree $d$. 
The following assertions are equivalent. \\ 
1) $\Phi$ is $\nu$-minimal and $\nu$-irreducible. \\
2) For every $f$, $g$ in $K_{d-1}[\rchi]$ with ${\rm deg}(f)+{\rm deg}(g)<[K[\rchi]\!:\! K]$ 
and every $h$ in $K_{d'-1}[\rchi]$, 
we have (*): $\nu(fg+h\Phi)=\min(\nu(fg),\nu(h\Phi))$. \\ 
3) $\Phi$ is a w key polynomial such that the sequence $(\Phi^m)$ ($m d<[K(\rchi)\!:\! K]$) 
is separated over $K_{d-1}[\rchi]$. 
\end{proposition}
\begin{proof} 
2) $\Rightarrow$ 1). Assume that $\Phi$ satisfies the hypothesis of 2). 
By setting $g=1$ in (*) if follows that 
$\Phi$ is a non constant monic polynomial in $K[\rchi]$ of degree $d$  such that for every 
$f$ in $K_{d-1}[\rchi]$ and every $h$ in $K_{d'-1}[\rchi]$ we have 
$\nu(f+h\Phi)=\min(\nu(f),\nu(h\Phi))$. Hence $\Phi$ is 
$\nu$-minimal. In order to prove that $\Phi$ is $\nu$ irreducible, let $f$ and $g$ 
in $K[\rchi]$ such that ${\rm deg}(f)+{\rm deg}(g)<[K[\rchi]\!:\! K]$ and 
for every $h \in K_{d'-1}[\rchi]$: (**) $\nu(f-h\Phi)=\min(\nu(f),\nu(h\Phi))$ and 
$\nu(g-h\Phi)=\min(\nu(g),\nu(h\Phi))$. By euclidean division, $f$ and $g$ can be written 
as $f=q\Phi+r$ and $g=q'\Phi+r'$. 
By setting $h=q$ in (**), we have $\nu(r)=\nu(f-q\Phi)=
\min (\nu(f),\nu(q\Phi))\leq \nu(f)$. 
By (*) we have $\nu(f)=\nu(r\!\cdot\! 1+q\Phi)
=\min(\nu(r\!\cdot\! 1,\nu(q\Phi))\leq \nu(r)$. Therefore $\nu(f)=\nu(r)$. 
In the same way, $\nu(g)=\nu(r')$. Let $h\in K_{d'-1}[\rchi]$. Then $\nu(fg-h\Phi)=
\nu((qq'\Phi+qr'+q'r-h)\Phi+rr')=\min(\nu((qq'\Phi+qr'+q'r-h)\Phi),\nu(rr'))$, by (*). 
Hence $\nu(fg-h\Phi)\leq \nu(rr')=\nu(fg)$ and $\min (\nu(fg),\nu(
h\Phi))\leq \nu(fg-h\Phi)\leq \nu(fg)$. So, $\nu(fg-h\Phi)=\min (\nu(fg),\nu(h\Phi))$. 
Hence $\Phi$ is $\nu$-irreducible. \\ 
\indent 1) $\Rightarrow$ 2). We assume that $\Phi$ is monic, $\nu$-minimal and 
$\nu$-irreducible. Let $f$, $g$ in $K_{d-1}[\rchi]$ 
 with ${\rm deg}(f)+{\rm deg}(g)<[K[\rchi]\!:\! K]$. Since $\Phi$ is 
$\nu$-minimal, for every $h$ in $K_{d'-1}[\rchi]$ we have $\nu(f+h\Phi)=\min(\nu(f),\nu(h\Phi))$ and 
$\nu(g+h\Phi)=\min(\nu(g),\nu(h\Phi))$. Now, $\Phi$ is $\nu$-irreducible, hence 
$\nu(fg+h\Phi)=\min(\nu(fg),\nu(h\Phi))$. \\
\indent 2) $\Rightarrow$ 3) has been proved in Lemma \ref{lmequivdefs}. \\ 
\indent 3) $\Rightarrow$ 2). We take $f$, $g$ in $K_{d-1}[\rchi]$ and $h$ in $K_{d'-1}[\rchi]$ 
with ${\rm deg}(f)+{\rm deg}(g)<[K[\rchi]\!:\! K]$. The 
polynomial $h$ can be written as $h=h_m\Phi^m+\cdots+h_\Phi+h_0$, where 
$h_m,\dots, h_1,h_0$ belong to $K_{d-1}[\rchi]$. Let $q:=q(f,g)$ and $r:=r(f,g)$; 
since $f$ and $g$ belong to $K_{d-1}[\rchi]$, we have deg$(q)<d$, i.e.\ 
$q\in K_{d-1}[\rchi]$. We have: $\nu(fg+h\Phi)=\nu(h_m\Phi^{m+1}+\cdots+ h_1\Phi^2+
(q+h_0)\Phi+r)=\min(\nu(h_n\Phi^{m+1}),\dots,\nu(h_1\Phi^2),\nu((q+h_0)\Phi),\nu(r))$. 
Since $\Phi$ is a w key polynomial, we have $\nu(fg)=\nu(r)\leq\nu(q\Phi)$. If 
$\nu(h_0\Phi)\geq \nu(r)$, then $\nu((q+h_0)\Phi)\geq \nu(r)$. Hence 
$\min (\nu((q+h_0)\Phi),\nu(r))=\nu(r)=\min (\nu(h_0\Phi),\nu(r))$. 
If $\nu(h_0\Phi)<\nu(r)$, then $\min (\nu((q+h_0)\Phi),\nu(r))=\nu(h_0\Phi)=
\min (\nu(h_0\Phi),\nu(r))$. 
Therefore: 
$\nu(fg+h\Phi)=\min(\nu(h_m\Phi^{m+1}),\dots,\nu(h_1\Phi^2),\nu(h_0\Phi),\nu(fg))=
\min(\nu(h\Phi),\nu(fg))$. 
\end{proof}
\indent 
We use the hypothesis ``$\nu$ is a pm valuation'' for proving 
2) $\Rightarrow$ 1). For proving 2) $\Rightarrow$ 3) the condition 
``for every $f\in K_{d'-1}[\rchi]$, $\nu(\Phi f)=\nu(\Phi)+\nu(f)$'' is 
sufficient. The remainder of the proof remains true with a $K$-module valuation. \\[2mm] 
\indent Proposition \ref{equivdefs} shows that the definition of MacLane is 
stronger than the definition of w key polynomials. Now, we extend the definition of S.\ MacLane 
to $K$-module valuations. 
\begin{definition} Let $\Phi$ be a polynomial of degree $d$. We say that $\Phi$ is a 
{\it ML key polynomial for} $\nu$ if $\Phi$ is a w key polynomial such that the sequence 
$(\Phi^m)$ ($md<[K(\rchi)\!:\! K]$) 
is separated over $K_{d-1}[\rchi]$.
\end{definition}
\indent If $\nu$ is a pm valuation, then it follows from Proposition \ref{avprop32} 
that $\Phi$ is a ML key polynomial for $\nu$ if, and only if, $\Phi$ is a w key polynomial for $\nu$ 
and $\nu=\nu_{\Phi}$. This shows the difference between w key polynomials and MacLane key polynomials. 
\section{Properties of key degrees}\label{section4} 
\indent In this section, we list some properties of key degrees. 
In particular, we give sufficient conditions for existing a finite number of key degrees and 
a characterization of the successor of 
a key degree. This characterization takes place in the construction of complete families of key polynomials. 
In this construction, in the case of an independent key degree, we will choose monic polynomials with 
maximal valuation. However, 
the fact that the degree $d$ of a key polynomial is independent does not imply that its valuation is maximal 
in $\nu(\rchi^d-K_{d-1}[\rchi])$. Indeed, 
we saw in Proposition \ref{rks33} 1) that every monic polynomial of degree $1$ is a key polynomial. 
This holds whether $1$ is an independent key degree or not. \\
\indent Except for Proposition\ref{rk322}, all properties of this section remain true if $\nu$ is a 
pseudo-valuation (and $\rchi$ is transcendental).   
In this case we assume that the degrees are lower than 
the degree of a generator of the ideal $\{f\in K[\rchi]\; :\; \nu(f)=\infty\}$ (in the settings, we can take 
$d<[K[\rchi]\!:\!K]$ in place of $d<{\rm deg}(\Pi)$). 
\subsection{Independent key degrees}\label{subsection35} 
We start this subsection with some observations about 
monic polynomials with maximal valuations. 
\begin{proposition}\label{aplm318} 
Let $d$ be an integer and $\nu$ be a pm valuation on $K[\rchi]$. \\ 
1) Let $\Phi$ be a w key polynomial of degree $d$. Then every 
monic polynomial $\Phi'$ 
of degree $d$, such that $\nu(\Phi')>\nu(\Phi)$, is a key polynomial. Therefore, if $d$ is not a key degree, 
then $\nu(\Phi)$ is maximal in $\nu(\rchi^d-K_{d-1}[\rchi])$. \\ 
2) Assume that $d$ is an independent key degree and let $\Phi$ be a key polynomial of degree $d$. 
Then every monic 
polynomial $\Phi'$ of degree $d$, such that $\nu(\Phi')\geq \nu(\Phi)$, is a key polynomial. 
In particular, any monic polynomial $\Phi'$ of degree $d$, such that $\nu(\Phi')$ is maximal in 
$\nu(\rchi^d-K_{d-1}[\rchi])$, is a key polynomial. 
\end{proposition}
\indent 
This proposition is a consequence of the following lemma. 
\begin{lemma}\label{lm318} 
Let $d$ be a positive integer, $\nu$ be a pm valuation on $K[\rchi]$, 
$\Phi$ and $\Phi'$ be monic polynomials of degree $d$ such that $\Phi$ is a w key polynomial for 
$\nu$. \\
1) If $\nu(\Phi')>\nu(\Phi)$, then $\Phi'$ is a w key polynomial for $\nu_{\Phi}$, 
a key polynomial for $\nu$, and $\nu_{\Phi}\leq \nu_{\Phi'}$. \\ 
2) If $\nu(\Phi')=\nu(\Phi)$ and $\Phi$ is a key polynomial for $\nu$, then 
$\Phi'$ is a key polynomial for both of $\nu$ and 
$\nu_{\Phi}$. Furthermore, $\nu_{\Phi}=\nu_{\Phi'}$.\\
3) If $\nu(\Phi')=\nu(\Phi)$ and $\Phi$, $\Phi'$ are 
w key polynomials for $\nu$, then $\nu_{\Phi}=\nu_{\Phi'}$. 
\end{lemma}
\begin{proof} Set 
$h=\Phi'-\Phi$. Assume that $\nu(\Phi')\geq\nu(\Phi)$. Then we have $\nu(h)\geq\nu(\Phi)$. Furthermore, 
$\Phi$ and $\Phi'$ are monic polynomials, so the degree of $h$ is lower than $d$. 
Now, $\Phi'=\Phi+h$, hence 
$\nu_{\Phi}(\Phi')=\min (\nu(\Phi),\nu(h))=\nu(\Phi)$. Let $f$, $g$ in $K_{d-1}[\rchi]$ with 
${\rm deg}(f)+{\rm deg}(g)<[K[\rchi]\!:\! K]$. 
We have: $\nu_{\Phi}(fg)=\nu_{\Phi}(f)+
\nu_{\Phi}(g)=\nu(f)+\nu(g)=\nu(fg)$. Let $r'=r_{\Phi'}(f,g)$, $q'=q_{\Phi'}(f,g)$. \\ 
\indent 1) Since $\nu(\Phi')>\nu(\Phi)$, we have $\nu(h)=\nu(\Phi)$ 
($=\nu_{\Phi}(\Phi')$). In order to get a contradiction, assume that 
$\nu_{\Phi}(q'\Phi')<\nu_{\Phi}(r')\;(=\nu(r'))$.  
Then $\nu_{\Phi}(fg)=\nu_{\Phi}(q'\Phi'+r')=\nu_{\Phi}(q'\Phi')=\nu(q'\Phi)=\nu_{\Phi}(q'\Phi)$. Now, 
$\nu(q'\Phi')>\nu(q'\Phi)=\nu(fg)$ and $\nu(r')>\nu_{\Phi}(q'\Phi')=\nu_{\Phi}(fg)=\nu(fg)$. 
Hence $\nu(fg)\geq \min (\nu(q'\Phi'),\nu(r'))>\nu(fg)$: a contradiction. Therefore, 
$\nu_{\Phi}(q'\Phi')\geq \nu_{\Phi}(r')$. Now, $\nu(r')=\nu_{\Phi}(r')$ and $\nu_{\Phi}(q'\Phi')=
\nu(q'\Phi)<\nu(q'\Phi')$. Hence $\Phi'$ is a key polynomial for $\nu$ (so $\nu(fg)=\nu(r')$). Since 
$\nu_{\Phi}(fg)=\nu(fg)=\nu(r')=\nu_{\Phi}(r')$, $\Phi'$ is a w key polynomial for 
$\nu_{\Phi}$. Furthermore, 
$\nu_{\Phi}$ and $\nu_{\Phi'}$ are equal on $K_{d-1}[\rchi]$, and $\nu_{\Phi'}(\Phi)=
\nu_{\Phi'}(\Phi'-h)=\min(\nu_{\Phi'}(\Phi'),\nu_{\Phi'}(h))=
\min(\nu(\Phi'),\nu(h))=\nu(h)=\nu(\Phi)=\nu_{\Phi}(\Phi)$. 
By Remark \ref{rekav312}, we have $\nu_{\Phi}\leq \nu_{\Phi'}$. \\ 
\indent 2) We assume that $\nu(\Phi')=\nu(\Phi)$ and $\Phi$ is a key polynomial for $\nu$. 
Let $q_1=q_{\Phi}(q',h)$ and $r_1=r_{\Phi}(q',h)$. 
Since $\Phi$ is a key 
polynomial for $\nu$, we have $\nu(q_1\Phi)>\nu(r_1)=\nu(q'h)\geq \nu(q'\Phi)$. Hence 
$\nu(q')<\nu(q_1)$ and $\nu(q'+q_1)=\nu(q')$. We have: 
$fg=q'\Phi'+r'=q'\Phi+q'h+r'=(q'+q_1)\Phi+r'+r_1$, hence $q'+q_1=q_{\Phi}(f,g)$, $r'+r_1=r_{\Phi}(f,g)$ and 
$\nu(fg)=\nu(r'+r_1)<\nu((q'+q_1)\Phi)=\nu(q'\Phi)\leq
\nu(r_1)$. It follows: $\nu(r'+r_1)=\nu(r')$, so $\nu(fg)=\nu(r')<\nu(q'\Phi)=\nu(q'\Phi')$. This 
proves that $\Phi'$ is a key polynomial for $\nu$. Now, $\nu_{\Phi}(fg)=\nu(fg)=
\nu(r')=\nu_{\Phi}(r')$ and $\nu_{\Phi}(q'\Phi')=\nu_{\Phi}(q'\Phi)=\nu(q'\Phi)>\nu(r')$. 
Hence $\Phi'$ is a key polynomial for $\nu_{\Phi}$. In the same way as in 1), we have: 
$\nu_{\Phi}\leq\nu_{\Phi'}$. Now, since $\Phi'$ is a key polynomial, we have in 
a symmetric way: $\nu_{\Phi'}\leq\nu_{\Phi}$.\\
\indent 3) We assume that $\nu(\Phi')=\nu(\Phi)$ and $\Phi$, $\Phi'$ are 
w key polynomials for $\nu$. 
We have: $\Phi'=\Phi+h$, with $\nu(h)\geq \nu(\Phi)$ and 
deg$(h)<d$. Hence $\nu_{\Phi}(\Phi')=\min(\nu(\Phi),\nu(h))=
\nu(\Phi)=\nu(\Phi')=\nu_{\Phi'}(\Phi')$. So by 
Remark \ref{rekav312} we have: $\nu_{\Phi'}\leq \nu_{\Phi}$. In the same way, 
$\nu_{\Phi}\leq \nu_{\Phi'}$, hence $\nu_{\Phi'}= \nu_{\Phi}$. 
\end{proof}
\indent  The following proposition gives a characterization of valuational key degrees by means of $\nu K_{d-1}[\rchi]$. 
\begin{proposition}\label{avpropdivb} Let $\nu$ be a pm valuation on $K[\rchi]$ and $d$ be a positive integer 
such that $1\leq d\leq [K(\rchi)\!:\!K]/2$. Then, $d$ is a valuational key degree if, and only if, 
$\nu K_{d-1}[\rchi]$ is a group and $\nu K_d[\rchi]\neq \nu K_{d-1}[\rchi]$. 
If this holds, then every monic polynomial $\Phi$ of degree $d$ 
such that $\nu(\Phi)$ is the maximum of $\nu(\rchi^d-K_{d-1}[\rchi])$ 
(which is equivalent to saying that $\nu(\Phi)\notin \nu K_{d-1}[\rchi]$) 
is a key polynomial. 
\end{proposition}
\begin{proof} $\Rightarrow$ follows from the definition and Corollary \ref{cor39}. 
Assume that $\nu K_{d-1}[\rchi]$ is a group and $\nu K_d[\rchi]\neq \nu K_{d-1}[\rchi]$.
By hypothesis, there is a polynomial $\Phi$ of degree $d$ such that 
$\nu(\Phi)\notin \nu K_{d-1}[\rchi]$. Since $\nu K_{d-1}[\rchi]$ is a group, by 
dividing $\Phi$ by an element of $K$ we can assume that $\Phi$ is a monic polynomial. 
Let $f\in K_{d-1}[\rchi]$. Then  
$\nu(f)\neq \nu(\Phi)$. Hence $\nu(\Phi-f)=\min (\nu(\Phi),\nu(f))\leq \nu(\Phi)$. 
Consequently, $\nu(\Phi)$ is the maximum of $\nu (\rchi^d-K_{d-1}[\rchi])$. \\
\indent 
Let $f$, $g$ in $K_{d-1}[\rchi]$, $q=q_{\Phi}(f,g)$ and $r=r_{\Phi}(f,g)$. Since 
$\nu(q\Phi)\notin\nu K_{d-1}[\rchi]$, we have that $\nu(q\Phi)\neq \nu(r)$. 
Hence $\nu(fg)=\min (\nu(q\Phi),\nu(r))$. Now, since $\nu K_{d-1}[\rchi]$ is a group, 
$\nu(fg)=\nu(f)+\nu(g)\in \nu K_{d-1}[\rchi]$. 
Hence $\nu(fg)\neq \nu(q\Phi)$. So, $\nu(q\Phi)>\nu(r)$. This proves that 
$\Phi$ is a key polynomial. 
Since $\nu(\phi)$ is the maximum of $\nu (\rchi^{d}-K_{d-1}[\rchi])$, $d$ is an independent 
key degree. Therefore it is a valuational key degree. 
\end{proof}
\begin{corollary}\label{avpropdivc} Let $\nu$ be a pm valuation on $K[\rchi]$ and $d$ be a positive integer. 
Assume that $\nu K_{d-1}[\rchi]$ is a group and let $d'$ be the smallest integer such that 
$\nu K_{d'}[\rchi]\neq \nu K_{d-1}[\rchi]$ (if any). 
Then $d'$ is a valuational key degree. In particular, the smallest degree $d$ such that 
$\nu K_d[\rchi]\neq \nu K$ (if any) is a valuational key degree.  
\end{corollary}
\subsection{Maximal key degrees}\label{subsect42} 
If $\rchi$ is algebraic over $K$, then there is a finite number of key degrees. We saw that if 
$\nu$ is a pseudo-valuation, then the degrees of the key polynomials are bounded by the degree of a 
generator of the ideal of elements with infinite valuation. 
We give sufficient conditions for a key degree being the maximal one. 
We start with a lemma. 
\begin{lemma}\label{laplm318} Let $\nu$ be a pm valuation on $K[\rchi]$ and $\Phi$ be 
a monic polynomial of degree $d$. Then, $\nu=\nu_{\Phi}$ on 
$K_d[\rchi]$ if, and only if, $\nu(\Phi)=\max \nu(\rchi^d-K_{d-1}[\rchi])$. \end{lemma}
\begin{proof} By the definition of $\nu_{\Phi}$, $\nu$ and $\nu_{\Phi}$ are equal on $K_{d-1}[\rchi]$. 
Hence, we can 
consider polynomials of degree $d$. So, we let $f$ be a polynomial of degree $d$. 
Without loss of generality we can assume that $f$ is a monic polynomial. Hence $f-\Phi$ 
has degree less than $\Phi$. 
Assume that $\nu(f)=\nu_{\Phi}(f)$. Then $\nu(f)=\nu_{\Phi}(\Phi+f-\Phi)=
\min (\nu(\Phi),\nu(f-\Phi))\leq \nu(\Phi)$. This proves that $\nu(\Phi)$ is the 
maximum of $\nu(\rchi^d-K_{d-1}[\rchi])$. Conversely, assume that 
$\nu(\Phi)=\max \nu(\rchi^d-K_{d-1}[\rchi])$. We have $\nu(f-\Phi)\geq\min (\nu(f), 
\nu(\Phi))=\nu(f)$. Therefore, $\nu(f)\geq \nu_{\Phi}(f)=\nu_{\Phi}(\Phi+f-\Phi)=
\min(\nu(\Phi),\nu(f-\Phi))\geq \nu(f)$. Hence $\nu_{\Phi}(f)=\nu(f)$.
\end{proof}
\begin{remark}\label{prop319}
Let $d$ be a positive integer, $\nu$ be a pm valuation on $K[\rchi]$. \\ 
1) It follows from Proposition \ref{aplem33} that if $\nu=\nu_{\Phi}$ for some w key polynomial $\Phi$, 
then there is no key degree greater than $d$. \\
2) Assume that $d$ is an independent key degree such that there is no key degree greater than $d$, and 
$\Phi$ is a key polynomial of degree $d$, with $\nu(\Phi)=\max (\nu(\rchi^d-K_{d-1}[\rchi]))$. 
By Lemma \ref{laplm318}, $\nu$ and $\nu_{\Phi}$ coincide on 
$K_d[\rchi]$. Now, by Proposition \ref{aplem33}, $\nu=\nu_{\Phi}$ on $K[\rchi]$. \end{remark}
\begin{proposition}\label{rk322} Let $\nu$ be a multiplicative valuation on $K(\rchi)$. 
If the rational rank of $\nu K(\rchi)$ over $\nu K$ is $1$, or the transcendence 
degree of $K(\rchi)_{\nu}$ over $K_{\nu}$ is $1$ (in other words, if $\rchi$ is transcendental and 
Abhyankar's inequality is an 
equality), then there is a finite number of  key degrees, 
and $\nu=\nu_{\Phi}$ for some w key polynomial $\Phi$. 
\end{proposition} 
\begin{proof} Abhyankar's inequality states that 
the transcendence degree of $K(\rchi)|K$ is at least equal to the sum of the transcendence degree of 
$(K(\rchi))_{\nu}|K_{\nu}$ and the rational rank of $\nu(K[\rchi])|\nu K$. 
Assume that the rational rank of $\nu K(\rchi)$ over $\nu K$ is $1$. 
Let $d$ be the smallest   integer such that there exists a monic polynomial $\Phi$ of degree $d$ such that 
$\nu(\Phi)$ is not rational over $\nu K$. If $d=1$, then we know that $\Phi$ is a 
key polynomial. Otherwise, clearly, $\nu(\Phi)\neq \nu(\rchi^d)=d\nu(\rchi)$. If 
$\nu(\Phi)< \nu(\rchi^d)$, then $\nu(\rchi^d-\Phi)=\nu(\Phi)$, with deg$(\rchi^d-\Phi)<d$: a contradiction 
(since $d$ is minimal). Hence $\nu(\Phi)>\nu(\rchi^d)$. Then $\nu(\Phi)=\max \nu(\rchi^d-K_{d-1}[\rchi])$. 
Let $f$, $g$ in $K_{d-1}[\rchi]$ and $q=q(f,g)$, $r=r(f,g)$. Since $d$ is minimal, we have:  
$\nu(r)\neq \nu(q\Phi)$  and $\nu(q\Phi)\neq \nu(f)+\nu(g)=\nu(fg)$. It follows that 
$\nu(fg)=\nu(r)<\nu(q\Phi)$ (the other cases lead to a contradiction). So $\Phi$ is a key 
polynomial for $\nu$. Now, the elements $\nu(\Phi^k)=k\nu(\Phi)$ are pairwise non-congruent 
modulo $\nu K_{d-1}[\rchi]$. In the same way as in Proposition \ref{BauranglaisS5}, this 
implies that the family $(\Phi^k)$ is separated over $K_{d-1}[\rchi]$. Therefore, 
$\nu=\nu_{\Phi}$. Now, by 1) of Remark \ref{prop319}, if $d'>d$, then $d$ is not a 
key degree. \\
\indent Now, assume that the transcendence degree of $K(\rchi)_{\nu}$ over $K_{\nu}$ is $1$. 
Let $\Phi$ be a polynomial such that $\nu(\Phi)=0$. If $in_{\nu}(\Phi)$ is algebraic over $G_{\nu}(K)$, then 
by Lemma \ref{irrpol} $in_{\nu}(\Phi)$ is algebraic over $K_{\nu}$ (the converse is trivial). 
Consequently, $in_{\nu}(\Phi)$ is transcendental over $K_{\nu}$ if, and only if, it transcendental over 
$G_{\nu}(K)$.  In particular, there is 
$\Phi$ in $K[\rchi]$ such that $in_{\nu}(\Phi)$ is transcendental over $G_{\nu}(K)$. 
Let $d$ be the smallest integer such that there exists a polynomial $\Phi$ of degree $d$ 
such that $in_{\nu}(\Phi)$ is transcendental over $G_{\nu}(K)$. Without loss of generality we can 
assume that $\Phi$ is a monic polynomial. 
Let $f$, $g$ in 
$K_{d-1}[\rchi]$, $q=q_{\Phi}(f,g)$, $r=r_{\Phi}(f,g)$. We have: 
$\nu(r)>\nu(fg)=\nu(q\Phi)\Leftrightarrow \nu(fg-q\Phi)>\nu(fg)=\nu(q\Phi)
\Leftrightarrow 0=in_{\nu}(fg)-in_{\nu}(q\Phi)\Leftrightarrow in_{\nu}(fg)=in_{\nu}(q\Phi)$. 
Now, $in_{\nu}(fg)=in_{\nu}(f)in_{\nu}(g)$ is algebraic over $G_{\nu}(K)$. 
Hence $in_{\nu}(fg)\neq in_{\nu}(q\Phi)$, which is transcendental over $G_{\nu}(K)$. It follows that 
$\nu(r)\leq \min (\nu(fg),\nu(q\Phi))\leq \nu(fg)$. Now, $in_{\nu}(r)\neq in_{\nu}(q\Phi)$, hence in the 
same way we get $\nu(fg)\leq \min (\nu (q\Phi),\nu(r))\leq \nu(r)$. Consequently, $\nu(fg)=\nu(r)$. 
This proves that $\Phi$ is a w key polynomial. 
Let $f_0,\dots,f_n$ in $K_{d-1}[\rchi]\backslash \{0\}$. By hypothesis, $in_{\nu}(f_0), \dots,in_{\nu}(f_n)$  
are algebraic over $G_{\nu}(K)$. 
Hence $in_{\nu}(f_n)in_{\nu}(\Phi)^n+\cdots+in_{\nu}(f_0)\neq 0$. Without loss of generality we can assume that 
$in_{\nu}(f_0)= \cdots =in_{\nu}(f_n)$. 
It follows that $\nu(f_0+\cdots +f_n\Phi^n)=\min( \nu(f_0),\cdots ,\nu(f_n\Phi^n))$.
This proves that the sequence $(\Phi^n)$ is separated over $K_{d-1}[\rchi]$. Therefore, $\nu=\nu_{\Phi}$. 
By 1) of Remark \ref{prop319}, if $d'>d$, then $d'$ is not a key degree. 
\end{proof} 
\indent Now we turn to a particular case of immediate key degrees that we can call {\it dense} key degrees. It follows that 
if a key degree is dense, then it is the greatest key degree. 
\begin{proposition}\label{prop49} 
Let $d\geq 2$, $\nu'\leq\nu$ be pm valuations on $K[\rchi]$. 
Assume that $\rchi^d$ is limit over $(K_{d-1}[\rchi],\nu)$, that the restrictions of $\nu$ and 
$\nu'$ to $K_{d-1}[\rchi]$ are equal, and that there exists a monic polynomial $\Phi$ of 
degree $d$ such that $\nu'(\Phi)<\nu(\Phi)$. Let $(\Phi_i)$ be 
a sequence of monic polynomials of degree $d$ such that the sequence $(\nu(\Phi_i))$ is increasing 
and cofinal in $\nu K[\rchi]$, and $\nu(\Phi)<\nu(\Phi_i)$. Then. \\
1) $d$ is an immediate key degree for $\nu$, $\Phi$ is a w key polynomial for $\nu$ and $\nu'$, and 
$\nu'=\nu'_{\Phi}$. \\
2) The $\Phi_i$'s are key polynomials for $\nu$, and for every $f\in K[\rchi]$ the sequence 
$(\nu_{\Phi_i}(f))$ is eventually equal to $\nu(f)$. \\
3) The extension $(K[\rchi]|K_{d-1}[\rchi],\nu)$ is dense. \\
4) There is no key degree greater than $d$. 
\end{proposition} 
\begin{proof} 1) By Proposition \ref{lem33}, $\Phi$ is a w key polynomial for $\nu'$, $\nu'=\nu_{\Phi}'$, 
and $\Phi$ is a key polynomial for $\nu$. Hence $d$ is a key degree for $\nu$. 
Since $\nu(\rchi^d-K_{d-1}[\rchi])=\nu K[\rchi]$, $d$ is an immediate key degree. \\
\indent 2) Since deg$(\Phi_i-\Phi)<d$, we have $\nu(\Phi_i-\Phi)=\nu'(\Phi_i-\Phi)$, and 
$\nu'(\Phi_i)=\nu_{\Phi}'(\Phi_i)=\nu_{\Phi}'(\Phi+\Phi_i-\Phi)=\min(\nu'(\Phi),\nu'(\Phi_i-\Phi))
=\min(\nu'(\Phi),\nu(\Phi_i-\Phi))=\min(\nu'(\Phi),\nu(\Phi))=\nu'(\Phi)<\nu(\Phi)<\nu(\Phi_i)$. 
By Proposition \ref{lem33}, $\Phi_i$ is a key polynomial for $\nu$. 
We have also: $\nu'=\nu_{\Phi_i}'$. Let $f\in K[\rchi]$, $k$ be the quotient 
of the euclidean division of $\mbox{deg}(f)$ by $d$, $\lambda = \nu'(f)-k\nu'(\Phi)$, and $i$ be 
such that $\nu(\Phi_i)> \max(\nu(f)-\lambda,\nu(f))$. 
Decompose $f$ as $f=f_{i,k}\Phi_i^k+\cdots+f_{i,1}\Phi_i+f_{i,0}$, where the $f_{i,j}$'s 
belong to $K_{d-1}[\rchi]$. Then $\nu'(f)=
\nu_{\Phi_i}'(f)=\min(\nu'(f_{i,j}\Phi_i^j))$. Therefore, for every $j$ we have: 
$\nu(f_{i,j})= \nu'(f_{i,j})\geq \nu'(f)-j\nu'(\Phi_i)
\geq \nu'(f)-k\nu'(\Phi_i)=\nu'(f)-k\nu'(\Phi)=\lambda$. 
Then for every $j$ we have $\nu(f_{i,j}\Phi_i^j)=\nu(f_{i,j})+j\nu(\Phi_i)>
\lambda +j\max(\nu(f)-\lambda,\nu(f))$. If $j\geq 1$, then $\nu(f_{i,j}\Phi^j)>\nu(f)$. 
It follows that $\nu(f)=\nu(f_{i,0})=\nu_{\Phi_i}(f)$. 
This proves that the sequence $(\nu_{\Phi_i}(f))$ is eventually equal to $\nu(f)$. \\
\indent 3) If $\nu(\Phi_i)> \max(\nu(f)-\lambda,\nu(f))$, then 
$$\nu(f-f_{i,0})=\nu\left(\sum_{j=1}^kf_{i,j}\Phi_i^j\right)
\geq \min_{1\leq j\leq k}(\nu(f_{i,j})+j\nu(\Phi_i))\geq\lambda+\nu(\Phi_i)$$ 
is cofinal in $\nu K[\rchi]$, since the sequence $(\nu(\Phi_i ))$ is. 
It follows that $(K[\rchi]|K_{d-1}[\rchi],\nu)$ is dense. \\
\indent 4) Let $f$ be a monic polynomial of degree $d'>d$. Then, there exists 
$i$ such that $\nu(f)=\nu_{\Phi_i}(f)$. By Proposition \ref{aplem33} 1), $f$ is not a  
key polynomial for $\nu$. Hence $d'$ is not a key degree. 
\end{proof}
\subsection{Successors of key degrees} 
The following two theorems characterize the successor of a key degree. They will be useful for 
proving the equivalence of the definitions of key polynomials and abstract key polynomials. They are also 
involved in the construction of complete families of key polynomials. 
\begin{theorem}\label{thm4xx} Let $\nu$ be a pm valuation on $K(\rchi)$, 
$d$ be an immediate key degree, and let $(\Phi_i)$ be 
a sequence of key polynomials of degree $d$ such that the sequence $(\nu(\Phi_i))$ is increasing 
and cofinal in $\nu(\rchi^d- K_{d-1}[\rchi])$. Then. \\
1) The sequence $(\nu_{\Phi_i})$ converges to $\nu$ on $K_d[\rchi]$ in the sense that 
for every $f\in K_d[\rchi]$ the sequence $(\nu_{\Phi_i}(f))$ is eventually equal to $\nu(f)$.  \\
2) Let $d'$ be the smallest degree (if any) such that there exists a monic polynomial $\Phi'$ of degree $d'$ 
satisfying $\nu_{\Phi_i}(\Phi')<\nu(\Phi')$ for every $i$. Then, $\Phi'$ is a key polynomial 
and $d'$ is the next key degree. \\
3) The extension $(K_{d'-1}[\rchi]|K_{d-1}[\rchi],\nu)$ is immediate. 
\end{theorem}
\begin{proof} 1) The family $(\nu_{\Phi_i})$ is increasing, and for every $f\in K[\rchi]$ 
we have $\nu_{\Phi_i}(f)\leq\nu(f)$. 
Assume that $f$ is a polynomial of degree $d$. Without loss of generality we can assume that 
$f$ is monic. Since the sequence $(\nu(\Phi_i))$ is cofinal in $\nu(\rchi^d-K_{d-1}[\rchi])$, 
there is an $i$ such that $\nu(\Phi_i)>\nu(f)$. Now, for eve\-ry $i$ such that $\nu(\Phi_i)>\nu(f)$ 
we have $\nu(f-\Phi_i)=\nu(f)$, and 
$\nu_{\Phi_i}(f)=\min (\nu(\Phi_i),\nu(f-\Phi_i))=\nu(f)$. \\
\indent 2) Let $f$, $g$ in $K_{d'-1}[\rchi]$ with ${\rm deg}(f)+{\rm deg}(g)<[K[\rchi]\!:\!K]$ and 
$q=q_{\Phi'}(f,g)$, $r=r_{\Phi'}(f,g)$. By hypothesis there exists $i$ such that $\nu_{\Phi_i}(f)=
\nu(f)$, $\nu_{\Phi_i}(g)=\nu(g)$, $\nu_{\Phi_i}(q)=\nu(q)$ and 
$\nu_{\Phi_i}(r)=\nu(r)$. Then $\nu(q\Phi'+r)=\nu(fg)=\nu(f)+\nu(g)=\nu_{\Phi_i}(f)+
\nu_{\Phi_i}(g)=\nu_{\Phi_i}(fg)=\nu_{\Phi_i}(q\Phi'+r)$, with $\nu(q\Phi')>\nu_{\Phi_i}(q\Phi')$. 
Assume that $\nu(q\Phi')\leq \nu(r)$. Hence $\nu_{\Phi_i}(q\Phi')<\nu(r)=\nu_{\Phi_i}(r)$, 
and $\nu(fg)=\nu_{\Phi_i}(fg)=\nu_{\Phi_i}(q\Phi')<\nu(q\Phi')=\min (\nu(q\Phi'),\nu(r))\leq 
\nu(fg)$: a contradiction. Hence $\nu(q\Phi')>\nu(r)$, which proves that $\Phi'$ is a key polynomial 
for $\nu$. Let $f$ be a monic polynomial of degree $d''<d'$. Then, there exists 
$i$ such that $\nu(f)=\nu_{\Phi_i}(f)$. By Proposition \ref{aplem33} 1), $f$ is not a key polynomial. \\
\indent 3) Let $f$ be a monic polynomial of degree $n$, $d<n\leq d'-1$. We show that $\nu(f)$ is not 
the maximum of $\nu(\rchi^n-K_{n-1}[\rchi])$. It will follow by Proposition \ref{prop111} 1) that the 
extension $(K_{d'-1}[\rchi]|K_{d-1}[\rchi],\nu)$ is immediate. Let $\Phi_i$ be a w key polynomial 
of degree $d$ such that $\nu(f)=\nu_{\Phi_i}(f)$, and let $f$ be written as 
$f=f_k\Phi_i^k+\cdots+f_1\Phi_i+f_0$, where $f_0,f_1,\dots,f_k$ belong to 
$K_{d-1}[\rchi]$, $f_k$ is monic, and ${\rm deg}(f_k)+kd=n$. Then 
$\nu(f)=\min (\nu(f_k\Phi_i^k),\dots,\nu(f_1\Phi_i),\nu(f_0))$. Assume that $\nu(f_0)>\nu(f)$. 
Then $\nu(f)=\min (\nu(f_k\Phi_i^k),\dots,\nu(f_1\Phi_i))$. Let $\Phi_j$ be a w key polynomial of degree $d$ 
such that $\nu(\Phi_j)>\nu(\Phi_i)$, and set $g=f_k\Phi_j^k+\cdots+f_1\Phi_j$. Then $g$ is a monic 
polynomial of degree $n$, and $\nu(g)\geq \nu_{\Phi_j}(g)=
\min (\nu(f_k\Phi_j^k),\dots,\nu(f_1\Phi_j))>\nu(f)$. Now, assume that $\nu(f_0)= \nu(f)$. Then 
$\nu(f-f_0)\geq\nu(f_0)$. If $\nu(f-f_0)>\nu(f_0)$, then we can let $g=f-f_0$. 
If $\nu(f-f_0)=\nu(f_0)=\nu(f)$, then we let $\Phi_j$ be a w key polynomial of degree $d$ 
such that $\nu(\Phi_j)>\nu(\Phi_i)$, and  $g=f_k\Phi_j^k+\cdots+f_1\Phi_j$. Then 
$\nu(g)\geq \nu_{\Phi_j}(g)=\min (\nu(f_k\Phi_j^k),\dots,\nu(f_1\Phi_j))>\min (\nu(f_k\Phi_i^k),
\dots,\nu(f_1\Phi_i))\geq\nu(f)$. 
\end{proof}
\indent Recall that if $M$ is a $K$-submodule of $K(\rchi)$, then we denote by 
$G_{\nu}(M)$ the additive group  $\displaystyle{\bigoplus_{\gamma\in \nu M} M_{\gamma,\nu}}$. 
\begin{theorem}\label{propdiv} 
Let $\nu$ be a pm valuation on $K(\rchi)$, $d$ be an independent key degree, 
$\Phi$ be a monic polynomial of degree $d$ such that $\nu(\Phi)$ is the maximum of 
$\nu(\rchi^d-K_{d-1}[\rchi])$. Let $d'$ be the next key degree if it exists 
or $d'=[K[\rchi]\! :\! K]$ otherwise. \\
1) $\Phi$ is a key polynomial, and  
$d'$ is the smallest degree such that there exists a monic polynomial $\Phi'$ of 
degree $d'$ with $\nu_{\Phi}(\Phi')<\nu(\Phi')$. \\
2) The restrictions of $\nu$ and $\nu_{\Phi}$ to $K_{d'-1}[\rchi]$ are equal.  \\
3) If $d<[K[\rchi]\! :\! K]/2$, 
then $G_{\nu}(K_{d'-1}[\rchi])$ is a graded algebra and the fraction 
$d'/d$ is equal to the dimension of the 
$G_{\nu}(K_{d-1}[\rchi])$-module $G_{\nu}(K_{d'-1}[\rchi])$. \\
4) If $(K_{d-1}[\rchi]|K,\nu)$ is vs defectless, then $(K_{d'-1}[\rchi]|K,\nu)$ is vs defectless. 
\end{theorem}
\begin{proof} 1) Since $\nu(\Phi)$ is the maximum of $\nu(\rchi^d-K_{d-1}[\rchi])$, by 2) of 
Proposition \ref{aplm318}, $\Phi$ is a key polynomial. 
By Proposition \ref{aplem33}, if $d'<[K[\rchi]\! :\! K]$, then 
$d'$ is the smallest degree such that there exists a 
monic polynomial $\Phi'$ of degree $d'$ with $\nu_{\Phi}(\Phi')<\nu(\Phi')$. 
If $d'=[K[\rchi]\! :\! K]$, then $\nu=\nu_{\Phi}$, and 
by letting $\Phi'$ be the 
minimal polynomial of $\rchi$ over $K$, we have $\nu(\Phi')=+\infty>\nu_{\Phi}(\Phi')$. 
Hence the assertion remains true. 
It also remains true in the case where $d'=+\infty$. \\
\indent  2) Since $\nu_{\Phi}\leq \nu$, by 1) the restrictions of $\nu$ and $\nu_{\Phi}$ 
to $K_{d'-1}[\rchi]$ are equal.\\
\indent 3) 
By Lemma \ref{avpropdiva} $in_{\nu}(\Phi)\notin G_{\nu}(K(\rchi))$. 
Since $d\leq [K[\rchi]\!:\!K]/2$, by Remark \ref{rkap36} $G_{\nu}(K_{d-1}[\rchi])$ is a graded algebra. 
Assume that $in_{\nu}(\Phi)$ is transcendental over $G_{\nu}(K_{d-1}[\rchi])$. 
Then $\rchi$ is transcendental over $K$, so $[K(\rchi)\!:\! K]$ is infinite. 
In the same way as in the proof of Propositions  
\ref{BauranglaisS5} and \ref{rk322} the family $(\Phi^n)$ is separated over $K_{d-1}[\rchi]$ and $\nu=\nu_{\Phi}$. 
Furthermore, the dimension of the $G_{\nu}(K_{d-1}(\rchi))$-module $G_{\nu}(K(\rchi))$ is also infinite. 
Assume that $d'=+\infty$. Then by 1) $\nu=\nu_{\Phi}$, 
the family $(\Phi^n)$ is separated over $K_{d-1}[\rchi]$ and $\nu=\nu_{\Phi}$. 
So the dimension of the $G_{\nu}(K_{d-1}(\rchi))$-module $G_{\nu}(K(\rchi))$ is also infinite. 
Now we assume that $d'<+\infty$. By 1) $\nu=\nu_{\Phi}$ on $K_{d'-1}[\rchi]$ and there exists a monic polynomial 
$\Phi'$ of degree $d'$ such that $\nu_{\Phi}(\Phi')<\nu(\Phi)$. Set $\Phi'=
f_n\Phi^n+f_{n-1}\Phi^{n-1}+\cdots +f_0$, with $f_n,\dots, f_0$ in $K_{d-1}[\rchi]$ (and $f_n\neq 0$). 
We prove that $f_n=1$. 
Since $\Phi$ is irreducible, there exists $g\in K_{d-1}[\rchi]$ 
such that $gf_n$ is congruent to $1$ modulo $\Phi$, i.e.\ $r_{\Phi}(gf_n)=1$. For $0\leq k\leq n$ 
set $r_k=r_{\Phi}(gf_k)$ ($r_n=1$), $q_k=q_{\Phi}(gf_k)$ and $\Phi''=\Phi^n+r_{n-1}\Phi^{n-1}+\cdots +r_0$. 
We have $$g(f_n\Phi^n+f_{n-1}\Phi^{n-1}+\cdots +f_0)=q_n\Phi^{n+1}+q_{n-1}\Phi^n+\cdots +q_0\Phi+
\Phi^n+r_{n-1}\Phi^{n-1}+\cdots +r_0.$$
Since $\Phi$ is a key polynomial, for every $k$ we have $\nu(r_k\Phi^k)=\nu(gf_k\Phi^k)<\nu(q_k\Phi^{k+1})$. 
Hence $$\nu(q_n\Phi^{n+1}+q_{n-1}\Phi^n+\cdots +q_0\Phi)\geq \min (\nu(q_n\Phi^{n+1}),\nu(q_{n-1}\Phi^n), 
\dots,\nu(q_0\Phi))>$$ 
$$\min (\nu(\Phi^n),\nu(r_{n-1}\Phi^{n-1}),\dots ,\nu(r_0))= \nu(g)+
\min (\nu(f_n\Phi^n),\nu(f_{n-1}\Phi^{n-1}),\dots ,\nu(f_0))=\nu(g)+\nu_{\Phi}(\Phi').$$ 
Therefore 
$$\nu(\Phi^n+r_{n-1}\Phi^{n-1}+\cdots +r_0)=$$
$$\nu(g(f_n\Phi^n+f_{n-1}\Phi^{n-1}+\cdots +f_0)-
q_n\Phi^{n+1}+q_{n-1}\Phi^n+\cdots +q_0\Phi)\geq$$ 
$$\min (\nu(g(f_n\Phi^n+f_{n-1}\Phi^{n-1}+\cdots +f_0)),
\nu(q_n\Phi^{n+1}+q_{n-1}\Phi^n+\cdots +q_0\Phi))>$$
$$\nu(g)+\min (\nu(f_n\Phi^n),\nu(f_{n-1}\Phi^{n-1}),\dots ,\nu(f_0))=$$
$$\min (\nu(gf_n\Phi^n),\nu(gf_{n-1}\Phi^{n-1}),\dots ,\nu(gf_0))=$$
$$\min (\nu(\Phi^n),\nu(r_{n-1}\Phi^{n-1}),\dots ,\nu(r_0))=\nu_{\Phi}(\Phi^n+r_{n-1}\Phi^{n-1}+\cdots +r_0).$$ 
So, $\nu(\Phi'')>\nu_{\Phi}(\Phi'')$, where $\mbox{deg}(\Phi'')=d'-\mbox{deg}(g)$.  
Since $d'$ is minimal, we have $g=1$, $f_k=1$ and $d'=nd$. \\
\indent We turn to $G_{\nu}(K_{d'-1}[\rchi])$. 
Let $k_1<\cdots<k_j$ be the integers such that 
$$\nu(f_{k_i}\Phi^{k_i})=
\min (\nu(f_n\Phi^n),\nu(f_{n-1}\Phi^{n-1}),\dots ,\nu(f_0))\mbox{, for }i\in \{1,\dots,j\}.$$ 
By properties of valuations we have 
$$\nu(f_{k_1}\Phi^{k_1}+\cdots+f_{k_j}\Phi^{k_j})>\min (
\nu(f_{k_1}\Phi^{k_1}),\dots,\nu(f_{k_j}\Phi^{k_j})).$$ 
In the same way as above, we can replace $f_{k_j}$ by $1$ 
in such an equality, and since $d'$ is minimal we have $n=k_j$. 
Therefore, we can assume that for $k\in \{0,\dots,n\}$ we have either $\nu(f_k\Phi^k)=\nu_{\Phi}(\Phi')$ 
or $f_k=0$. Consequently $in_{\nu}(\Phi)^n+in_{\nu}(f_{n-1})in_{\nu}(\Phi)^{n-1}+\cdots +in_{\nu}(f_0)=0$. 
Now, for every $g_{n-1},\dots,g_0$ in $K_{d-1}[\rchi]$ we have 
$\nu(g_{n-1}\Phi^{n-1}+\cdots+g_0)=\min (\nu(g_{n-1}\Phi^{n-1}),\cdots,\nu(g_0))$. 
Hence $in_{\nu}(g_{n-1})in_{\nu}(\Phi)^{n-1}+\cdots+in_{\nu}(g_0)\neq 0$. 
It follows that 
$X^n+in_{\nu}(f_{n-1})X^{n-1}+\cdots +in_{\nu}(f_0)$ is 
the minimal polynomial of $in_{\nu}(\Phi)$ over $G_{\nu}(K_{d-1}[\rchi])$. 
This also shows that the group $G_{\nu}(K_{d'-1}[\rchi])$ is equal to 
$G_{\nu}(K_{d-1}[\rchi])(in_{\nu}(\Phi))$, so it is a subalgebra of $G_{\nu}(K(\rchi))$. 
Since for every $g_{n-1},\dots,g_0$ in $K_{d-1}[\rchi]$, 
$in_{\nu}(g_{n-1})in_{\nu}(\Phi)^{n-1}+\cdots+in_{\nu}(g_0)\neq 0$,  
$in_{\nu}(1),in_{\nu}(\Phi),\dots,in_{\nu}(\Phi^{n-1})$ are linearly independent over 
$G_{\nu}(K_{d-1}[\rchi])$. 
Therefore $(in_{\nu}(1),in_{\nu}(\Phi),\dots,in_{\nu}(\Phi)^{n-1})$ is a basis of the 
$G_{\nu}(K_{d-1}[\rchi])$-module $G_{\nu}(K_{d'-1}[\rchi])$. Hence 
its dimension is $n$. \\ 
\indent 4) By Proposition \ref{prop214}, it is sufficient to show that for every integer $n\leq d'-1$, 
$\nu(\rchi^n-K_{n-1}[\rchi])$ has a maximum. Since $(K_{d-1}[\rchi]|K,\nu)$ is vs defectless, 
we can assume that $n\geq d$ (Proposition \ref{prop214}). 
Let $f=f_k\Phi^k+\cdots+f_1\Phi+f_0\in \rchi^n-K_{n-1}[\rchi]$, where $f_0,f_1,\dots,f_k$ belong to 
$K_{d-1}[\rchi]$, $f_k$ is monic, and ${\rm deg}(f_k)+kd=n$. Then 
$\nu(f)=\min (\nu(f_k\Phi^k),\dots,\nu(f_1\Phi),\nu(f_0))\leq \nu(f_k\Phi^k)$. Let $j=n-dk$ be the 
degree of $f_k$. Since $j\leq d-1$, there exists $g\in \rchi^j-K_{j-1}[\rchi]$ such that $\nu(g)$ is the maximum of 
$\nu(\rchi^j-K_{j-1}[\rchi])$. Then $\nu(f)\leq \nu(g\Phi^k)$, which proves that 
$\nu(g\Phi^k)$ is the maximum of $\nu(\rchi^n-K_{n-1}[\rchi])$. 
\end{proof}
\subsection{Complete families of key polynomials.}\label{subsection34}
\indent Let $\Phi_1,\;\dots,\; \Phi_n,\; \dots$ be monic irreducible 
polynomials of $K[\rchi]$ of degrees $d_1=1,\; d_2,\; \dots,\; d_n,\;$ respectively, 
where $d_{n-1}$ divides $d_n$ (if the family has a maximal element $d_n$, then we set 
$d_{n+1}=[K(\rchi)\!:\! K]$). We let $\mathcal{B}$ be 
the family of the $\Phi_1^{e_1}\cdots \Phi_n^{e_n}$, where for $0\leq i\leq n-1$ 
we have: 
$\displaystyle{0\leq e_i< \frac{d_{i+1}}{d_i}}$. Since the mapping $\Phi_i^{e_i}\mapsto {\rm deg}\left(\Phi_i^{e_i}\right)$ 
is one-to-one 
from $\mathcal{B}$ onto $[0,[K(\rchi)\!:\! K][$, $\mathcal{B}$ is a basis of the $K$-module $K[\rchi]$. 
Furthermore, for every $m$ in $\NN$, $\mathcal{B}\cap K_m[\rchi]$ is a basis of the 
$K$-module $K_m[\rchi]$. Now, we can define a basis even if some degree does not divide the 
following one. Indeed, in the case where we have only $d_1=1<d_2<\cdots<d_n<\cdots$, 
we require: for every $n$, $e_1+e_2d_2+\cdots+e_nd_n<d_{n+1}$. We say that 
$\mathcal{B}$ is {\it generated by polynomials} if it is constructed in the above way. If so, then 
$\Phi_1,\;\Phi_2,\; \dots,\; \Phi_n,\; \dots$, are called the {\it generating polynomials} for 
$\mathcal{B}$. 
\begin{remark} \label{rk324}
Let $\nu$ be a $K$-module valuation on $K(\rchi)$, 
$\Phi_1,\dots,\;\Phi_k,\dots$ be 
generating polynomials for a basis $\mathcal{B}$. For $k\geq 1$, fix $\nu'(\Phi_{k})$. 
For every $e_1,\dots,e_k$ with $e_1+e_2d_2+\cdots+e_kd_k<d_{k+1}$, 
let $\nu'(\Phi_{1}^{e_1}\cdots\Phi_{k}^{e_k})=
e_1\nu'(\Phi_{1})+e_2\nu'(\Phi_{2})+\cdots+e_k\nu'(\Phi_{k})$ and for every 
pairwise distinct $y_1,\dots,y_k$ in $\mathcal{B}$, $x_1,\dots,x_k$ in $K$ set 
$\displaystyle{\nu'(x_1y_1+\cdots+x_ky_k)=\min_{1\leq i\leq k}\nu(x_i)\nu'(y_i) }$. Then 
$\nu'$ is a vs defectless $K$-module valuation. If for every $k\geq 1$ we have 
$\nu'(\Phi_k)=\nu(\Phi_k)$, then 
 the $K$-module valuation $\nu'$ defined above is the $K$-module valuation 
$\nu_{\Phi_{d_1},\dots,\Phi_{d_k}}$ defined in Notations \ref{notat325}. 
\end{remark}
\indent 
Let $\nu$ be a pm valuation on $K(\rchi)$, $1=d_1<d_2<\cdots<d_k<\cdots$ be 
the sequence of key degrees of $\nu$. Let ${\mathcal F}$ be a family of key polynomials.  
We say that ${\mathcal F}$ is a {\it complete family of key polynomials for} $\nu$ 
if $\nu$ is the supremum of the family of the pm valuations $\nu_{\Phi_{d_1},\dots,\Phi_{d_k}}$, where 
$k\in \NN\backslash\{0\}$, for $1\leq i\leq k$ $\Phi_{d_i}\in \mathcal{F}$, deg$(\Phi_{d_i})=d_i$. This means that 
if $f$ is a polynomial of degree $d$, then $\nu(f)$ is equal to the maximum of the family 
$\nu_{\Phi_{d_1},\dots,\Phi_{d_k}}(f)$, where $d_k$ is the greatest key degree at most equal to $d$. \\[2mm]
\indent For every pm valuation on $K[\rchi]$ there exists a complete family of key polynomials 
(see for example  \cite{NS18}, \cite{DMS18}). 
We complete this subsection by the construction of a complete family of key polynomials 
which does not require to be constructed by induction on the degree. 
\begin{theorem}\label{thm46} 
Let $\nu$ be a pm valuation on $K(\rchi)$, $1=d_1<d_2<\cdots<d_k<\cdots$ be 
the sequence of  key degrees of $\nu$ and ${\mathcal F}$ be a family  of key polynomials for $\nu$  
which satisfies the following properties for every $k\geq 1$. \\
If $d_k$ is an independent key degree, then ${\mathcal F}$ contains exactly one key polynomial 
$\Phi_{d_k}$ of degree $d_k$, and $\nu(\Phi_{d_k)}=\max (\nu(\rchi^{d_k}-K_{d_k-1}[\rchi]))$. 
For notational convenience, for every integer $m$ we set $\Phi_{d_k,m}=\Phi_{d_k}$. \\ 
If $d_k$ is an immediate key degree, then the key polynomials of degree $d_k$ of ${\mathcal F}$ 
form a sequence $(\Phi_{d_k,m})$ such that the sequence $(\nu(\Phi_{d_k,m}))$ is increasing, cofinal in 
$\nu(\rchi^{d_k}-K_{d_k-1}[\rchi])$. \\ 
For every $k$, $m_1,\dots,m_k$, we let $\nu_{(m_1,\dots,m_k)} = 
\nu_{\Phi_{d_{1,m_1}},\dots,\Phi_{d_{k,m_k}}}$ (see Notations \ref{notat325} and 
Remark \ref{rk324}). 
Then $\mathcal{F}$ is a complete family of key polynomials. Furthermore, 
for every $f$, there are infinitely many $(m_1,\dots,m_k)$'s such that $(\nu_{(m_1,\dots,m_k)}(f))=\nu(f)$.  
\end{theorem} 
\begin{proof} 
Note that, for every $k$, $m_1,\dots,m_k$, we have $\nu_{(m_1,\dots,m_k)}\leq \nu$. 
Let $f\in K[\rchi]$ and $d_k$ be the greatest key degree such that the degree of $f$ is at least 
equal to $d_k$. By Theorems \ref{propdiv} and \ref{thm4xx}, there exists $m_k$ 
such that $\nu(f)=\nu_{\Phi_{d_k,m_k}}(f)$. Furthermore, since the family $(\nu_{\Phi_{d_k,m_k}})$ 
is increasing, we have $m_k'\geq m_k\Rightarrow \nu_{\Phi_{d_k,m_k'}}(f)=\nu(f)$. 
Let 
$f=f_j\Phi_{d_k,m_k}^j+\cdots+f_1\Phi_{d_k,m_k}+f_0$, 
where $f_0,\;f_1,\dots,\;f_j$ belong to $K_{d_k-1}[\rchi]$. We know that there is some 
$m_{k-1}$ such that $\nu(f_0)=\nu_{\Phi_{d_{k-1},m_{k-1}}}(f_0),\dots,\;\nu(f_j)=
\nu_{\Phi_{d_{k-1},m_{k-1}}}(f_j)$. Then, $\nu (f)=\nu_{\Phi_{d_{k-1},m_{k-1}},\Phi_{d_k,m_k}}(f)$. 
So by induction we get a $k$-uple $(m_1,\dots,m_k)$ such that 
$\nu(f)=\nu_{(m_1,\dots,m_k)}(f)$. \\ 
\indent Now, we can do the same construction with any $m_k'\geq m_k$, which proves that there are 
infinitely many such $k$-uples. 
\end{proof} 
\section{Abstract key polynomials}\label{subsec37}
\indent In this section $\rchi$ is assumed to be transcendental over $K$. The mapping $\nu$ is a valuation or 
a pseudo-valuation. If $\nu$ is a pseudo-valuation, then we assume that the degrees of $\Phi$ and $\Phi'$ are lower than 
the degree of a generator of the ideal $\{f\in K[\rchi]\; :\; \nu(f)=\infty\}$. Our goal is to prove that 
our definition of key polynomials is equivalent to the definition of abstract key polynomials. \\
\indent  We start with some notations. \\
\indent Let $f$ be a polynomial degree of $n$.  
For $1 \leq i \leq n$, set $f_{(i)}=(1/i!)\,f^{(i)}$, where 
$f^{(i)}$ is the $i$-th formal derivative of $f$. If the characteristic of $K$ is 
$p>0$, then we do not replace $p\! \cdot \! x$ by $0$; the simplification holds, if necessary, 
after the division by $i!$. For example, if 
$f(X)=X^p$, then we have $f_{(1)}(X)=pX^{p-1}=0$, and
$f_{(p)}(X)=(1/p!)\! \cdot\! p!=1$. \\
\indent 
For any polynomial $f$, let 
$\varepsilon_{\mu}(f)$ be the maximum of the set 
$$\displaystyle{\left\{\frac{\nu(f)-\nu(f_{(i)})}{i} \mid i\in \mathbb{N}\backslash\{0\}\right\} }.$$  
We say that $\Phi$ is an {\it abstract key polynomial} 
(or a {\it key polynomial}) {\it for }
$\nu$ if it is monic and for every $f\in K[\rchi]$ such that $\varepsilon_{\nu}(f)\geq \varepsilon_{\nu}(\Phi)$ we have 
${\rm deg}(f)\geq {\rm deg}(\Phi)$ (\cite{DMS18}).
\begin{remark}(\cite[Remark 2.1]{NS18})\label{rkNS21} 
The monic polynomials of degree $1$ are abstract key polynomials.  
\end{remark}
\begin{proposition}(\cite[Lemma 2.3 (iii)]{NS18}) Let $\Phi$ be an abstract key polynomial for $\nu$, $h_1,\dots, 
h_s$ be polynomials with degrees less than  ${\rm deg}(\Phi)$ and $q$ be the quotient of the euclidean 
division of $\displaystyle{\prod_{i=1}^sh_i}$ by $\Phi$. Then  $\displaystyle{\nu\left(\prod_{i=1}^sh_i-q\Phi\right)
=\nu\left(\prod_{i=1}^sh_i\right)<\nu(q\Phi)}$. 
\end{proposition}
\begin{corollary}\label{cor351} The abstract key polynomials are key polynomials as defined in Definitions \ref{def32}. \end{corollary}
\indent The converse is a consequence of the following theorem. In this theorem for a key polynomial $\Phi$ 
we let $\square (\Phi):=\{f\in K[\rchi] \mid f \mbox{ is monic }, \nu_{\Phi}(f)<\nu(f)\}$, and $\boxminus(\Phi)$ 
be the subset of elements of $\square(\Phi)$ with minimum degree. 
\begin{theorem}(\cite[Theorem 2.12]{NS18})\label{thmNS212} 
Let $\Phi$ be a monic polynomial. Then $\Phi$ is an abstract 
key polynomial if, and only if, there exists an abstract key polynomial $\Phi_{-}$ such that: \\ 
either a) $\Phi$ belongs to $\boxminus(\Phi_-)$,\\
or b) \\
i) the elements in $\boxminus(\Phi_-)$ have the same degree as $\Phi_-$, \\
ii) the set of valuations of all elements in $\boxminus(\Phi_-)$ has no maximum element, \\
iii) for every $\Phi'$ in  $\boxminus(\Phi_-)$, $\nu_{\Phi'}(\Phi)<\nu(\Phi)$, \\ 
iv) the degree of $\Phi$ is minimum among the degrees of polynomials satisfying iii). 
\end{theorem}
\begin{corollary}\label{cor353} The key polynomials as defined in Definitions \ref{def32} 
are abstract key polynomials. \end{corollary}
\begin{proof}
We prove this by induction on the degrees. By Proposition \ref{rks33} 1) and Remark \ref{rkNS21} 
this is true for polynomials of degree $1$. 
Let $\Phi$ be a key polynomial of degree $d>1$, and assume that every key polynomial of degree less that $d$ is
an abstract key polynomial. We denote by $d'$ te preceding key degree. If $d'$ is an independent key degree, then by a) of 
Theorem \ref{thmNS212} and Theorem \ref{propdiv}, $\Phi$ is an abstract key polynomial. 
 If $d'$ is an immediate key degree, then by b) of 
Theorem \ref{thmNS212} and Theorem \ref{thm4xx}, $\Phi$ is an abstract key polynomial. 
\end{proof}

\begin{thebibliography}{99}
%
%
\bibitem{AK65} J.\ Ax, S.\ Kochen, Diophantine problems over local fields 1 \& 2, 
Am.\ J.\ Math.\ 87 (1965) 605--630 \& 631--648. 
%
\bibitem{B1} W.\ Baur, Die Theorie der Paare reell abgeschlossener  
K\"orper, Logic and Algorithmic (in honour of E.\  Specker), Monographies de l'Enseignement Math\'ematique, 
n$^{\circ}$ 30 (1982), Universit\'e de Gen\`eve, Geneva, 25--34.
%
\bibitem{B2} W.\ Baur, On the theory of pairs of real closed fields, J.\ Symb.\ Logic 47 (1982) 669-679 (1982). 
%
\bibitem{BCK18} A.\ Blaszczok, P.\ Cubides Kovacsics, F-V.\ Kuhlmann,  On valuation independence 
and defectless extensions of valued fields, J. of Algebra, (2020) 555 69--95.  
%
%
\bibitem{DMS18} J.\ Decaup,  W.\ Mahboub, M.\ Spivakovski,  Abstract key polynomials and comparision 
theorems with the key polynomials of Mac Lane-Vaqui\'e, Illinois J.\ Math.\ 62 (2018) 253--270. 
%
%
\bibitem{D88} F.\ Delon, Extensions s\'epar\'ees et imm\'ediates de corps valu\'es, 
J.\ Symb.\ Logic 53 (1988) 421--428. 
%

\bibitem{D91} F.\ Delon, Ind\'ecidabilit\'e de la th\'eorie des paires de corps valu\'es 
henseliens, J.\ Symb.\ Logic 56 (1991) 1236-1242.
%
%
\bibitem{HOS07} F.J.\ Herrera Govantes, M.A.\ Olalla Acosta,  M.\ Spivakovsky, 
Valuations in algebraic field extensions, J.\  of Algebra 312 (2007) 
1033--1074.
%
%
\bibitem{K75} S.\ Kochen,  The model theory of local fields, Logic Conference, Kiel 1974, 
Lecture Notes in Mathematics (499) (1975), Springer Verlag, Berlin. 
%
%
\bibitem{FVK} F.-V.\ Kuhlmann,  Book on valuation theory (in preparation) Chapter 3. Available 
at  https://math.usask.ca/~fvk/Fvkbook.htm. 
%
\bibitem{L89} G.\ Leloup,  Th\'eories compl\`etes de paires de corps valu\'es henseliens, 
J.\ Symb.\ Logic 55 (1990) 323--339.
%
%
\bibitem{L03} G.\ Leloup,  Properties of extensions of algebraically maximal fields, 
Coll.\ Math.\ 54 (2003) 53--72.
%
\bibitem{L18} G.\ Leloup,  Key polynomials, separate and immediate valuations, 
and simple extensions of valued fields.  arxiv.org/abs/1809.07092, 
 https://hal.archives-ouvertes.fr/hal-01876056 (2018) (first and second versions). 
%
\bibitem{ML36a} S.\ MacLane, A construction for absolute values in polynomials 
rings, Trans.\  Amer.\  Math.\  Soc 40 (1936)  363--395.
%
\bibitem{ML36b} S.\ MacLane,  A construction for prime ideals as absolute values of 
an algebraic field, Duke Math.\  j.\  2 (1936) 492--510.
%
\bibitem{N19} J.\ Novakoski, Key polynomials and minimal pairs, J.\ Algebra 523 (2019) 1--14. 
%
\bibitem{NS18} J.\ Novakoski, M.\ Spivakovski,  Key polynomials and pseudo-convergent sequences, 
J.\ Algebra 495 (2018) 199--219. 
%
\bibitem{R} P.\ Ribenboim, Th\'eorie des valuations L,es Presses de 
l'universit\'e de Montr\'eal, Montr\'eal (1968).
%
\bibitem{V07}  M.\ Vaqui\'e, Extension d'une valuation,  Trans.\  Amer.\  Math.\  Soc.\  
359 7 (2007) 3439--3481. 
%
\end{thebibliography}
\end{document}